\documentclass[11pt]{amsart}
\usepackage[a4paper, margin=2cm, top=2.2cm, headsep=0.75cm, bottom=1.8cm, footskip=0.9cm]{geometry}             
\usepackage[dvipsnames]{xcolor}   		             		
\usepackage{graphicx}				
\usepackage{amssymb,mathrsfs}
\usepackage{amsthm,amsmath,stmaryrd}
\usepackage{tikz}
\usepackage{tikz-cd}
\usepackage{accents,upgreek, enumitem}
\usepackage{bm}
\usepackage{mathtools}
\usepackage[all]{xy}
\usepackage{caption}
\usepackage{array,booktabs}

\tikzset{
  commutative diagrams/.cd, 
  arrow style=tikz, 
  diagrams={>=stealth}
}

\DeclareMathAlphabet{\mathpzc}{OT1}{pzc}{m}{it}

\usepackage[backref]{hyperref}
\hypersetup{
  colorlinks   = true,          
  urlcolor     = blue,          
  linkcolor    = purple,          
  citecolor   = blue             
}

\usepackage[capitalize]{cleveref}

\makeatletter
\newtheorem*{rep@theorem}{\rep@title}
\newcommand{\newreptheorem}[2]{%
\newenvironment{rep#1}[1]{%
 \def\rep@title{#2 \ref*{##1}}%
 \begin{rep@theorem}}%
 {\end{rep@theorem}}}
\makeatother
\newtheorem{mainthm}{Theorem}
\crefname{mainthm}{Theorem}{Theorems}
\newreptheorem{mainthm}{Theorem}

\newtheorem{question}{Question}

\newtheorem{conj}[question]{Conjecture}
\newtheorem{theorem}{Theorem}[section]
\newtheorem{corollary}[theorem]{Corollary}
\newtheorem{lemma}[theorem]{Lemma}
\newtheorem{proposition}[theorem]{Proposition}

\theoremstyle{definition}
\newtheorem{definition}[theorem]{Definition}
\newtheorem{example}[theorem]{Example}

\theoremstyle{remark}
\newtheorem{rem1}[theorem]{Remark}
\newenvironment{remark}{\begin{rem1}\em}{\end{rem1}}

\title[Homological Lagrangian monodromy for some monotone tori]{Homological Lagrangian monodromy\\ for some monotone tori}

\author{Marcin Augustynowicz}
\address{Trinity College, Cambridge, CB2 1TQ, United Kingdom}
\email{\href{mailto:marcin.l.augustynowicz@gmail.com}{marcin.l.augustynowicz@gmail.com}}
\author{Jack Smith}
\address{St John's College, Cambridge, CB2 1TP, United Kingdom}
\email{\href{mailto:j.smith@dpmms.cam.ac.uk}{j.smith@dpmms.cam.ac.uk}}
\author{Jakub Wornbard}
\address{St John's College, Cambridge, CB2 1TP, United Kingdom}
\email{\href{mailto:jww43@cantab.ac.uk}{jww43@cantab.ac.uk}}

\newcommand{\HLM}{\mathcal{H}_L}
\newcommand{\HLMS}{\mathcal{SY}_L}
\renewcommand{\phi}{\varphi}

\newcommand{\Cl}{\operatorname{\mathnormal{C\kern-0.15ex \ell}}}
\newcommand{\Hess}{\operatorname{Hess}}
\newcommand{\Ham}{\operatorname{Ham}}
\newcommand{\Symp}{\operatorname{Symp}}

\newcommand{\id}{\operatorname{id}}
\newcommand{\GL}{\operatorname{GL}}

\newcommand{\Sym}{\operatorname{Sym}}
\newcommand{\U}{\mathrm{U}}
\newcommand{\SU}{\mathrm{SU}}

\newcommand{\CC}{\mathbb{C}}
\newcommand{\KK}{\mathbb{K}}
\newcommand{\PP}{\mathbb{P}}
\newcommand{\RR}{\mathbb{R}}
\newcommand{\ZZ}{\mathbb{Z}}
\newcommand{\QQ}{\mathbb{Q}}

\newcommand{\calL}{\mathcal{L}}
\newcommand{\diff}{\mathrm{d}}
\newcommand{\eps}{\varepsilon}

\newcommand{\rmf}{\mathrm{f}}
\newcommand{\rmt}{\mathrm{t}}
\newcommand{\rmd}{\mathrm{d}}
\newcommand{\rmv}{\mathrm{v}}
\newcommand{\rme}{\mathrm{e}}
\newcommand{\gp}[1]{\underline{\mathbf{#1}}}
\newcommand{\Bl}{\operatorname{Bl}}

\setenumerate{label=(\alph*)}

\begin{document}

\begin{abstract}
Given a Lagrangian submanifold $L$ in a symplectic manifold $X$, the homological Lagrangian monodromy group $\mathcal{H}_L$ describes how Hamiltonian diffeomorphisms of $X$ preserving $L$ setwise act on $H_*(L)$.  We begin a systematic study of this group when $L$ is a monotone Lagrangian $n$-torus.  Among other things, we describe $\mathcal{H}_L$ completely when $L$ is a monotone toric fibre, make significant progress towards classifying the groups than can occur for $n=2$, and make a conjecture for general $n$.  Our classification results rely crucially on arithmetic properties of Floer cohomology rings.
\end{abstract}

\maketitle


\section{Introduction}

\subsection{Homological Lagrangian monodromy}

The Hamiltonian diffeomorphism group $\Ham(X)$ of a symplectic manifold $(X, \omega)$ is a central object in symplectic topology and has been studied intensively.  Given a Lagrangian submanifold $L \subset X$, there is a natural relative version
\[
\Ham(X, L) = \{\phi \in \Ham(X) : \phi(L) = L\},
\]
which provides a bridge between the Hamiltonian dynamics of $X$ and the Floer theory of $L$, but it has received much less attention.

This group is infinite-dimensional and difficult to get a handle on---heuristically, its Lie algebra is the space of exact $1$-forms on $X$ that pull back to $0$ on $L$.  But one way to extract information about its group of connected components is through its action on the homology of $L$.  The image of this representation in $\GL(H_*(L))$ is the \emph{(Hamiltonian) homological Lagrangian monodromy group} of $L$, denoted $\HLM$.  All (co)homology groups are over $\ZZ$ unless indicated otherwise.

The group $\HLM$ was introduced by M.-L.~Yau in \cite{YauMonodromyAndIsotopy}, building on groundbreaking work of Chekanov \cite{Chekanov}, and the following is known:
\begin{itemize}
\item Yau \cite{YauMonodromyAndIsotopy}: For a monotone Clifford (product) torus in $\CC^2$, $\HLM$ lies in a subgroup of $\GL(H_*(L))$ isomorphic to $\ZZ/2$, proved using symplectic capacities.  This subgroup can be realised explicitly, by using $\U(2) \subset \Ham(\CC^2)$ to swap the two factors.  Yau also gives a similar argument for the Chekanov torus in $\CC^2$, using a different $\ZZ/2$ subgroup.
\item Chekanov \cite[Theorem 4.5]{Chekanov}: For an arbitrary (not necessarily monotone) Clifford or Chekanov torus $L$ in $\CC^n$, $\HLM$ comprises those automorphisms of $H_1(L)$ that preserve the Maslov class, area class, and certain `distinguished classes' in $H_1(L)$.  Chekanov's result is actually phrased in terms of symplectomorphisms, rather than Hamiltonian isotopies, but they can be upgraded to symplectic isotopies by \cite[Lemma 3.1(b)]{Chekanov} and then to Hamiltonian isotopies by flux considerations as in \cref{rmkFlux}.
\item Hu--Lalonde--Leclercq \cite{HuLalondeLeclercq}: $\HLM$ is trivial if $L$ is weakly exact.  The proof uses Floer cohomology and the relative Seidel morphism, and in general works with $H_*(L; \ZZ/2)$ in place of $H_*(L)$.  It can be upgraded to $\ZZ$ coefficients if $L$ is relatively spin and one restricts attention to those $\phi \in \Ham(X, L)$ that preserve some relative spin structure in a suitable sense.
\item Ono \cite[Section 4]{Ono}: For $(S^1)^n \subset D^2(a)^n$, $\HLM$ is trivial if $a \leq 2\pi$ and contains the symmetric group on the $n$ factors if $a > 2\pi$.  Here $S^1$ is the unit circle and $D^2(a)$ is the symplectic $2$-disc of area $a$.  The proof uses displacement energy to constrain $\HLM$, and with the methods of this paper it is straightforward to upgrade `contains' to `is equal to' in the $a > 2\pi$ case.
\item Ono \cite[Section 5]{Ono}: For the product of equators in $S^2 \times S^2$, where the two factors have equal area, $\HLM$ is generated by the rotations of each sphere through angle $\pi$ about a horizontal axis.  The proof uses Floer cohomology with local systems, and its Hamiltonian-invariance, and extends to the product of equators in $(S^2)^n$ for any $n$.
\item Mangolte--Welschinger \cite[Corollary 3.2]{MangolteWelschinger}: If $L \subset X$ is a Lagrangian $2$-torus in a uniruled symplectic $4$-manifold then $\HLM$ contains no hyperbolic element (meaning an element with real eigenvalues different from $\pm 1$).  The proof uses symplectic field theory and Gromov--Witten theory.
\item Keating \cite[Section 4.2]{Keating}: There are monotone Lagrangians in $\CC^3$ with $\HLM$ non-trivial.  Topologically they are of the form $S^1 \times \Sigma_g$, where $\Sigma_g$ is a surface appropriate genus $g$.
\end{itemize}
As far as we are aware, these are the only results on $\HLM$ in the literature.  Recently, Porcelli has studied the corresponding group for \emph{generalised} homology theories in the weakly exact setting \cite{Porcelli}.

The goal of this paper is to initiate a systematic study of $\HLM$ in the case where $L$ is a monotone Lagrangian torus.  Recall that a Lagrangian $L \subset X$ is \emph{monotone} if the image of the Maslov class $\mu \in H^2(X, L)$ in $H^2(X, L; \RR)$ is positively proportional to class of the symplectic form $[\omega]$.  This gives good compactness properties for moduli spaces of holomorphic curves with boundary on $L$, which will be crucial to our constraints.  Since the cohomology ring of a torus is generated in degree $1$, the action of $\Ham(X, L)$ on $H_*(L)$ is completely captured by its action on $H_1(L)$.  So we may think of $\HLM$ as a subgroup of $\GL(H_1(L)) \cong \GL(n, \ZZ)$, where $n$ is the dimension of $L$, and this is what we shall do.

\begin{remark}
\label{rmkCompact}
We assume throughout that $X$ is either compact or that there is some mechanism preventing holomorphic curves escaping to infinity (for example, $X$ is convex and cylindrical at infinity).  In the case of toric manifolds we will discuss a specific mechanism based on semiprojectivity.
\end{remark}

Along the way we also prove some results about the group $\HLMS$, defined in the same way as $\HLM$ but with $\Ham(X, L)$ replaced by
\begin{multline*}
\Symp_\infty(X, L) = \{\phi \in \Symp(X) : \phi(L) = L \text{, $\phi$ is compactly supported or otherwise respects} \\ \text{the mechanism preventing escape of holomorphic curves}\}.
\end{multline*}
We call this $\HLMS$ rather than $\mathcal{S}_L$ to distinguish it from the \emph{smooth monodromy group} studied in \cite{YauMonodromyGroups}.  Note that we needn't say anything about behaviour at infinity in the definition of $\HLM$, since we can always cut off the generating Hamiltonian of any $\phi \in \Ham(X, L)$ outside the region swept out by $L$ in order to make it compactly supported.  This also means that $\HLM$ is always a subgroup of $\HLMS$.


\subsection{Main results}

Fix a monotone Lagrangian $n$-torus $L \subset X$.  Given a class $\beta \in H_2(X, L)$ with $\mu(\beta) = 2$, we can count holomorphic discs with boundary on $L$ that represent the class $\beta$.  We denote this count by $n_\beta$, and define it precisely in \cref{sscDiscCounts}.  Let $B_1 \subset H_1(L)$ denote the set of boundaries of classes $\beta$ with $n_\beta \neq 0$, and let $r$ be the rank of the rational span of $B_1$ in $H_1(L; \QQ)$.

\begin{mainthm}
\label{Theorem1}
\begin{enumerate}
\item\label{Theorem1itm1} The action of $\HLMS$ on $H_1(L)$ permutes the elements of $B_1$.
\item\label{Theorem1itm2} If $r = n$ then the induced map $\HLMS \to \Sym(B_1)$ is injective, so $\HLMS$ is finite.
\item\label{Theorem1itm3} If $r = 0$ then $\HLM$ is trivial.
\end{enumerate}
A fortiori, the first two parts hold with the subgroup $\HLM$ in place of $\HLMS$.
\end{mainthm}

\begin{remark}
\label{rmkProductSwap}
The last part cannot be extended to $\HLMS$ in general.  Consider, for example, the product of the zero sections in $T^*S^1 \times T^*S^1$.  The group $\HLMS$ can swap the two factors but $\HLM$ cannot.  Here $\Symp_\infty(X, L)$ comprises those symplectomorphisms that preserve $L$ setwise and are exact and cylindrical at infinity.

If we relax our notion of monotonicity to the requirement that $\mu$ and $\omega$ are positively proportional as homomorphisms $\pi_2(X, L) \to \RR$ then we can simplify this example, as follows.  Take $L = K^2 \subset T^4$, where $K \subset T^2$ is a non-contractible simple closed curve.  Again $\HLMS$ can swap the two factors but $\HLM$ cannot.
\end{remark}

The proof of \ref{Theorem1itm1} uses symplectomorphism-invariance of the disc counts $n_\beta$, the main idea of which is standard but which requires some care to deal with signs, relying on a special feature of spin structures on tori (a priori, one could imagine two discs counting with the same sign being carried by a symplectomorphism to two discs counting with opposite signs).   From this, \ref{Theorem1itm2} is an immediate consequence since a linear automorphism of $H_1(L)$ is completely determined by its action on a spanning set.  The proof of \ref{Theorem1itm3}, meanwhile, adapts the argument used by Ono for the product of equators in $S^2 \times S^2$, and generalises the $L \cong T^n$ case of Hu--Lalonde--Leclercq's result.

Our second result concerns compact toric manifolds.  For us, a \emph{toric manifold} $X$ means a symplectic toric manifold, obtained from $(\CC^N, \omega_\mathrm{std})$ by symplectic reduction via a subtorus $T^{N-n}$ of $T^N$ acting in the obvious way.  This is described in more detail in \cref{sscToricBackground}, but for now we note that $X$ is encoded by a convex \emph{moment polytope} $\Delta$, and carries a residual Hamiltonian action of $T \coloneqq T^N / T^{N-n}$.  The vector space in which $\Delta$ lives is naturally identified with the dual $\mathfrak{t}^\vee$ of the Lie algebra of $T$, and under this identification $\Delta$ is precisely the image of the moment map for the $T$-action.  The fibres of the moment map over interior points of $\Delta$ are Lagrangian free $T$-orbits---\emph{(Lagrangian) toric fibres}---and at most one of them is monotone.

\begin{definition}
\label{defToricAut}
By a \emph{toric automorphism} of $X$ we mean a symplectomorphism $\phi$ of $X$, which is $T$-equivariant after twisting by a Lie group automorphism $\psi$ of $T$.  Explicitly, this means $\phi(t \cdot x) = \psi(t) \cdot \phi(x)$ for all $x \in X$ and $t \in T$.  If $X$ has a monotone fibre then it is naturally preserved setwise by toric automorphisms, so there is a homomorphism from the toric automorphism group to $\HLMS$.  By the \emph{toric Torelli group} of $X$ we mean the group of toric automorphisms acting trivially on $H^*(X)$.
\end{definition}

\begin{remark}
Toric manifolds as defined in this way are naturally K\"ahler manifolds, and moreover can be constructed algebro-geometrically as GIT quotients.  See \cite{ProudfootToric} for a discussion of this approach and its equivalence to the above.  In the algebro-geometric setting there is a natural notion of toric morphism, and the algebro-geometric toric automorphism group agrees with the toric automorphism group as in \cref{defToricAut}.
\end{remark}

With this setup in place we can give our second result, which describes $\HLM$ and $\HLMS$ for a monotone Lagrangian toric fibre $L$ in a compact toric manifold $X$.  We first state it geometrically and then give a more concrete combinatorial formulation.

\begin{repmainthm}{Theorem2}[Geometric version]
For a monotone toric fibre $L \subset X$, the natural homomorphism from the toric automorphism group of $X$ to $\HLMS$ is an isomorphism.  Its restriction to the toric Torelli group gives an isomorphism onto $\HLM$.
\end{repmainthm}

To give the combinatorial formulation, which is the version we will prove, let $A \subset \mathfrak{t}$ be the lattice given by the kernel of the exponential map.  Each facet (codimension-$1$ face) $F_1, \dots, F_N$ of $\Delta$ has a primitive normal vector $\nu_j$ in $A$.  There is a canonical identification between $H_1(L)$ and $A$, and between $H_2(X, L)$ and the abelian group $\ZZ^N$ freely generated by certain \emph{basic classes} $\beta_j$, such that $\partial \beta_j \in H_1(L)$ is identified with $\nu_j \in A$.  Write $K \subset \ZZ^N$ for the space of linear relators between the $\nu_j$.

It is well-known by work of Cho \cite{Cho} and Cho--Oh \cite{ChoOh} that
\[
n_\beta = \begin{cases} 1 & \text{if $\beta = \beta_j$ for some $j$} \\ 0 & \text{otherwise}.\end{cases}
\]
This means that $B_1 = \{\nu_1, \dots, \nu_N\}$, and we are in the situation of \cref{Theorem1}\ref{Theorem1itm2}, so $\HLM$ and $\HLMS$ can be viewed as subgroups of the symmetric group $S_N$ on the $\nu_j$.

\begin{mainthm}[Combinatorial version]
\label{Theorem2}
For a monotone toric fibre $L$, $\HLM$ and $\HLMS$ comprise those permutations of the $\nu_j$ that fix $K$ pointwise and setwise respectively.
\end{mainthm}

To deduce the geometric version from the combinatorial version, first note that the latter says $\HLMS$ is as large as it could possibly be: it contains every permutation of the $\nu_j$ that can be induced by a linear automorphism of $H_1(L)$.  This amounts to saying that $\HLMS$ is, via its dual representation on $H^1(L; \RR) \cong \mathfrak{t}^\vee$, the group of symmetries of the moment polytope $\Delta$, and this corresponds precisely to the toric automorphism group of $X$.  The fact that $\HLM$ corresponds to the toric Torelli group then follows from the fact that there is a natural bijection between $K$ and $H^2(X)$, which generates the whole cohomology ring of $X$, so fixing $K$ pointwise is equivalent to acting trivially on $H^*(X)$.

\begin{example}
\label{exBlowup}
Let $X$ be the monotone toric blowup of $\CC\PP^2$ at one point, and $L \subset X$ the unique monotone Lagrangian toric fibre.  We can set things up so that $\Delta$ is as shown in \cref{figDelta}, with $L$ the fibre over the point $0$.
\begin{figure}[ht]
\begin{tikzpicture}[blob/.style={circle, draw=black, fill=black, inner sep=0, minimum size=\blobsize}]
\def\blobsize{1mm}

\draw (0, 0) node[blob]{};
\draw (0, 0) node[anchor=east]{$0$};

\draw (-1, 0) -- (0, -1) -- (2, -1) -- (-1, 2) -- cycle;
\draw (-1, 1) [anchor=east]node{$F_3$};
\draw (-0.5, -0.5) [anchor=north east]node{$F_4$};
\draw (1, -1) [anchor=north]node{$F_1$};
\draw (0.5, 0.5) [anchor=south west]node{$F_2$};
\end{tikzpicture}
\caption{The moment polytope $\Delta$ of the monotone toric blowup of $\CC\PP^2$ at a point.\label{figDelta}}
\end{figure}
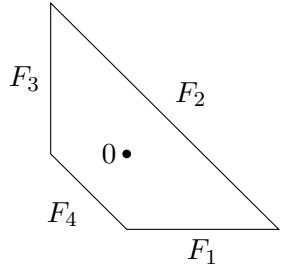
Here $F_1$ and $F_3$ correspond to the proper transforms of the two lines whose intersection we blew up, $F_2$ corresponds to a disjoint line, and $F_4$ corresponds to the exceptional divisor.  The space $K$ of relators is spanned by $\nu_1+\nu_2+\nu_3$ and $\nu_2+\nu_4$.  So $\HLM$ comprises those permutations of the $\nu_j$ that preserve both of these expressions, which is generated by the transposition $(1 \ 3)$ in $S_4$ that swaps $\nu_1$ and $\nu_3$.  In fact, any permutation preserving $K$ setwise must also preserve these two expressions, as they are the unique elements whose coefficients are $1, 1, 1, 0$ and $1, 1, 0, 0$ in some order, respectively.  So $\HLMS = \HLM$.
\end{example}

\begin{example}
For the product of equators in $(S^2)^n$, we can take $\Delta = [-1, 1]^n$.  This has $N = 2n$, with $\nu_{2j-1}$ the $j$th standard basis vector and $\nu_{2j}$ its negative.  The relators are $\nu_{2j-1}+\nu_{2j}$, so $\HLM$ is generated by the transpositions of pairs $\nu_{2j-1}$, $\nu_{2j}$, which correspond to rotating each factor through angle $\pi$ about a horizontal axis, as proved by Ono.  For $\HLMS$ we get $\HLM \rtimes S_n$, generated by $\HLM$ and permutations of the factors.
\end{example}

We discuss the extension of \cref{Theorem2} to a large class of non-compact toric varieties in \cref{sscNonCompactToric}.  The descriptions of $\HLM$ and $\HLMS$ are exactly the same as in the compact case, where $\Symp_\infty(X, L)$ means those symplectomorphisms $\phi$ that preserve $L$ setwise and are holomorphic with respect to the standard complex structure outside a compact set.

\begin{example}
For the Clifford torus in $\CC^n$, we can take $\Delta$ to be the non-negative orthant, with $N=n$ and $\nu_j$ the $j$th standard basis vector.  There are no linear relators, so $\HLM = \HLMS$ comprises all permutations of the factors, generalising Yau's result for $\CC^2$ and recovering the monotone version of Chekanov's result.
\end{example}

\begin{remark}
In this toric setting, the only properties of Hamiltonian diffeomorphisms used to derive the constraints are that they are symplectomorphisms and they act trivially on $H_2(X)$.  So we would get the same result if we replaced $\Ham(X, L)$ with, for example, the group
\[
\Symp_0(X, L) = \{\phi \in \Symp(X, L) \text{ or } \Symp_\infty(X, L) : \text{$\phi$ is symplectically isotopic to $\id_X$}\}.
\]
In non-toric settings the action of $\Ham(X, L)$ can be strictly smaller than that of $\Symp_0(X, L)$, at least if we relax the notion of monotonicity as in \cref{rmkProductSwap}.  For example, let
\[
X = S^2 \times S^2 \times S^1 \times [0, 1] \big/ (x, y, e^{i\theta}, 0) \sim (y, x, e^{i\theta}, 1),
\]
i.e.~the mapping torus of swapping the $S^2$ factors, with symplectic form $\omega_{S^2} \oplus \omega_{S^2} \oplus \rmd \theta \wedge \rmd t$, where $t$ is the coordinate on the $[0, 1]$ factor.  Let $L = (\text{equator}) \times (\text{equator}) \times S^1 \times \{0\}$.  Any class in $\pi_2(X, L)$ can be represented by one contained in $S^2 \times S^2 \times S^1 \times \{0\}$, so the weaker form of monotonicity holds.  By translating in the $t$-direction we see that swapping the first two $S^1$ factors of $L$ can be realised by a symplectic isotopy.  However, by the same argument as used by Ono for the product of equators in \cite[Section 5]{Ono} it cannot be realised by a Hamiltonian diffeomorphism.
\end{remark}

Our next topic of study is the classification of possible groups that can occur for $\HLM$ when $n=2$, and in light of \cref{Theorem1}\ref{Theorem1itm3} we restrict attention to the case $r > 0$.

\begin{mainthm}
\label{Theorem3}
Suppose $n=2$ and $r > 0$.  If $X$ is symplectically aspherical (meaning $\omega$ vanishes on $\pi_2(X)$) and simply connected then $\HLMS$ naturally embeds in the infinite dihedral group on
\[
\{\gamma \in H_1(L) : \text{$\mu(\beta) = 2$ for some, or equivalently all, $\beta \in \pi_2(X, L)$ with boundary $\gamma$}\}.
\]
The image is either trivial or is generated by a single involution, in which case there are two possibilities up to conjugation in $D_\infty$, namely involutions with or without a fixed point.
\end{mainthm}

The hypotheses are automatically satisfied if $X = \CC^2$; the $r > 0$ condition holds since if $r$ were $0$ then $HF^*(L, L)$ would be isomorphic to $H^*(L)$ by \cite[Proposition 6.1.4(2)]{biran2007quantum}, and hence $L$ would be non-displaceable.  There are thus at most three possibilities for each of $\HLMS$ and $\HLM$ when $X = \CC^2$.  Yau \cite{YauMonodromyAndIsotopy} shows that the Clifford and Chekanov tori realise the two non-trivial possibilities for $\HLM$, by describing explicit non-trivial Hamiltonian isotopies.  Given these isotopies, \cref{Theorem3} immediately reproves Yau's result that they generate all of $\HLM$, by a completely different route.

Yau asks \cite[Question 3.4]{YauMonodromyAndIsotopy} whether $\HLM \cong \ZZ/2$ for all monotone tori in $\CC^2$, which is a weak form of the (open) question of whether any monotone Lagrangian torus in $\CC^2$ is symplectomorphic or Hamiltonian isotopic to either a Clifford or Chekanov torus.  \Cref{Theorem3} reduces this to the following.

\begin{question}
Is $\HLM$ non-trivial for all monotone tori in $\CC^2$?
\end{question}

In fact, we do not know any examples with trivial $\HLM$ satisfying the hypotheses of \cref{Theorem3}, although there seems to be no reason why they should not exist.

\begin{question}
Can $\HLM$ ever be trivial in the setting of \cref{Theorem3}?
\end{question}

Returning now to a general monotone Lagrangian $2$-torus $L \subset X$, a choice of basis for $H_1(L)$ lets us view $\HLM$ as a subgroup of $\GL(2, \ZZ)$, and if we forget the basis then the subgroup is well-defined up to conjugacy.  Recall also (this is recapped in \cref{sscFloer}) that associated to $L$ is a \emph{superpotential} $W \in \ZZ[H_1(L)]$ which, again after a choice of basis for $H_1(L)$, we can view as a Laurent polynomial $W(x, y) \in \ZZ[x^{\pm1}, y^{\pm1}]$.  We prove the following results.

\begin{mainthm}  For $n=2$ we have:
\label{Theorem4}
\begin{enumerate}
\item\label{rk0} If $r = 0$ then $\HLM$ is trivial (by \cref{Theorem1}\ref{Theorem1itm3}).
\item\label{rk1} If $r = 1$ then, after a suitable choice of basis for $H_1(L)$, either:
\begin{itemize}
\item $W = a\pm x$ for some $a \in \ZZ$, and $\HLM$ is a subgroup of
\begin{equation}
\label{eqinfinite1}
\left\{ \begin{pmatrix} 1 & \ZZ \\ 0 & \pm1 \end{pmatrix} \right\} \subset \GL(2, \ZZ).
\end{equation}
\item $W = a \pm \big(x + \frac{1}{x}\big)$ for some $a \in \ZZ$, and $\HLM$ is a subgroup of
\begin{equation}
\label{eqinfinite2}
\left\{ \begin{pmatrix} \pm 1 & 2\ZZ \\ 0 & 1 \end{pmatrix} \right\} \subset \GL(2, \ZZ).
\end{equation}
\item $W = a \pm \big(x - \frac{1}{x}\big)$ for some $a \in \ZZ$, and $\HLM$ is a subgroup of
\begin{equation}
\label{eqinfinite3}
\left\{ \begin{pmatrix} 1 & 4\ZZ \\ 0 & 1 \end{pmatrix} \right\} \subset \GL(2, \ZZ).
\end{equation}
\item $\HLM$ is trivial.
\end{itemize}
\item\label{rk2} If $r = 2$ then $\HLM$ is finite (by \cref{Theorem1}\ref{Theorem1itm2}), so corresponds to one of the $13$ conjugacy classes of finite subgroups of $\GL(2, \ZZ)$.  Of these, $6$ are impossible, $5$ are realised by monotone Lagrangian toric fibres, and the remaining $2$ cannot occur torically but we are unable to rule them out in general.
\end{enumerate}
\end{mainthm}

Arithmetic properties of $W$ and the Floer cohomology rings play an important role in our arguments.  In particular, we consider Floer cohomology over $\ZZ$ and pass to fields of different characteristic to exploit divisibility of various structure constants.

\begin{remark}
Tonkonog proves in \cite[Lemma 3.3]{TonkonogSH} that if $X$ is compact then the constant term in the superpotential of any monotone torus $L \subset X$ must vanish.  So in this case the $a$ in \cref{Theorem4} must be zero.
\end{remark}

\begin{mainthm}
\label{Theorem5}
If $L$ is the product of the equator in $\CC\PP^1$ with the zero section in $T^*S^1$ then $W = x+\frac{1}{x}$ and $\HLM$ is \eqref{eqinfinite2}.  In contrast, if $L$ is the product of the unit circle in $\CC$ with the zero section in $T^*S^1$ then $W = x$ but $\HLM$ is trivial.  In both cases we are using the obvious product basis of $H_1(L)$.
\end{mainthm}

We use flux considerations both to construct non-trivial elements in the first case and to prove triviality in the second case.  We believe that result for $\CC\PP^1 \times T^*S^1$ represents the first known instance of elements of infinite order in $\HLM$ for monotone $L$.  Whilst editing this paper, we learnt from Brendel that he has independently discovered the same elements \cite{Brendel}, by completely different methods.

Some obvious questions remain unanswered.

\begin{question}
Do the two groups left open in \cref{Theorem4}\ref{rk2} arise?
\end{question}

We expect that the answer is yes, for the following reason.  These two groups---which are generated by rotations of $H_1(L)$ of order $2$ and $3$ respectively---are subgroups of other groups that definitely \emph{do} arise, isomorphic to the dihedral groups $D_4$ and $D_6$ of order $4$ and $6$.  So any argument that ruled them out would say `if $\HLM$ contains a rotation of order $2$ or $3$ then it must also contain a reflection', which would be a remarkable dynamical result.

\begin{question}
Do the groups \eqref{eqinfinite1} and \eqref{eqinfinite3} in \cref{Theorem4}\ref{rk1} arise?
\end{question}

In \cite{BackgroundCocycles} we show that \eqref{eqinfinite3} occurs for the product of the equator in $\CC\PP^1$ with the zero section in $T^*S^1$, equipped with a specific non-trivial relative spin structure $\mathfrak{s}$, where $\HLM$ now denotes those Hamiltonian diffeomorphisms respecting $\mathfrak{s}$ in a suitable sense.  For \eqref{eqinfinite1}, however, we do not have a good guess for what the answer should be!  \cref{Theorem5} shows that the obvious thing to try, namely the product Lagrangian in $\CC \times T^*S^1$, does not work.  More generally, in \cref{rmkFlux} we explain why if \eqref{eqinfinite1} arises for $L \subset X$ then $\pi_2(X)$ must contain an element of Chern number $1$, so first one must find such an $X$ containing a monotone torus with potential $a \pm x$.

Finally, we consider the classification problem in dimensions $n > 2$.  We focus on the case where $\HLM$ is finite (which holds if $r=n$ for example), and see what can be said about $\HLM$ as an abstract group, up to isomorphism.

\begin{mainthm}
\label{Theorem6}
If $n=3$ and $\HLM$ is finite then $\HLM$ is isomorphic to a subgroup of $S_4$, $S_3 \times S_2$, or $S_2 \times S_2 \times S_2$.
\end{mainthm}

These three groups arise for the monotone toric fibres in $\CC\PP^3$, $\CC\PP^2 \times \CC\PP^1$, and $(\CC\PP^1)^3$ respectively.  We are writing $\ZZ/2$ as $S_2$ for consistency with what follows.

As $n$ grows, the problem rapidly becomes much more complicated, but computer experiments suggest the following.

\begin{conj}
\label{conjHighDim}
If $\HLM$ is finite then either:
\begin{enumerate}
\item\label{conjitm1} $\HLM$ is isomorphic to a subgroup of $\GL(n-1, \ZZ)$.
\item\label{conjitm2} There exist integers $n_1, \dots, n_k \geq 2$ with $\sum (n_j - 1) = n$ such that $\HLM$ is isomorphic to a subgroup of $S_{n_1} \times \dots \times S_{n_k}$.
\end{enumerate}
\end{conj}

This can be verified directly for $n\leq 6$, as explained in \cref{sscHigherDim}.  When $L$ is a monotone toric fibre, \cref{Theorem2} implies that $\HLM$ is isomorphic to $S_{N_1} \times \dots \times S_{N_k}$ for some $N_j$ satisfying $\sum (N_j - 1) \leq n$, so case \ref{conjitm2} holds; see \cref{rmkConjectureToric}.  For each choice of $n_j$ in case \ref{conjitm2}, the full group $S_{n_1} \times \dots \times S_{n_k}$ arises for the monotone toric fibre in $\CC\PP^{n_1-1} \times \dots \times \CC\PP^{n_k-1}$.


\subsection{Structure of the paper}

In \cref{secGeneral} we discuss some generalities on holomorphic discs and Floer theory for monotone tori, and prove \cref{Theorem1}.  \cref{secToric} then recaps the basics of toric geometry and proves \cref{Theorem2}.  In \cref{secFurtherConstraints} we introduce two further tools for analysing $\HLM$, one based on Floer continuation elements and the other based on $1$-eigenspaces of $\HLM$ acting on $H_1(L)$.  \cref{secn2} focuses on the classification problem for $n=2$, and applies the tools developed earlier to prove \cref{Theorem3,Theorem4}.  We also discuss, in \cref{sscFloerConstraints}, a sort of converse question, namely: given $\HLM$, what can be said about the Floer theory of $L$?  We conclude the $n=2$ analysis with \cref{secThm5}, in which we consider $\CC \times T^*S^1$ and $\CC\PP^1 \times T^*S^1$ and prove \cref{Theorem5}.  The paper ends with \cref{secHigher}, where we discuss the case of $n > 2$ and prove \cref{Theorem6}.


\subsection{Acknowledgements}

We are grateful to Jo\'e Brendel for telling us about his work and for many helpful comments, and to the anonymous referees for their valuable suggestions.  JS also thanks Jonny Evans, Ailsa Keating, Noah Porcelli, Oscar Randal-Williams, Dhruv Ranganathan, Ivan Smith, Jake Solomon, and Umut Varolgunes for useful conversations.  This paper began in an undergraduate summer research project supervised by the second-named author.


\section{General considerations}
\label{secGeneral}

Fix throughout the rest of the paper a monotone Lagrangian $n$-torus $L$ in a symplectic manifold $(X, \omega)$, in which holomorphic curves cannot escape to infinity, as in \cref{rmkCompact}.


\subsection{Index $2$ disc counts}
\label{sscDiscCounts}

Take a class $\beta \in H_2(X, L)$ with $\mu(\beta) = 2$.  In this subsection we define the count $n_\beta$ of holomorphic discs in class $\beta$.  This is all standard, except perhaps for the discussion of invariance of orientations.

For an $\omega$-compatible almost complex structure $J$ on $X$, let $\mathcal{M}_1(J, \beta)$ denote the moduli space of $J$-holomorphic discs $u : (D^2, \partial D^2) \to (X, L)$ representing class $\beta$, with a single boundary marked point $z_0$, modulo reparametrisation.  This comes with an evaluation map $\operatorname{ev} : \mathcal{M}_1(J, \beta) \to L$ sending $(u, z_0)$ to $u(z_0)$.  If $J$ is \emph{regular}, which holds for a generic choice, then the moduli space is a smooth compact $n$-manifold.  Different choices of regular $J$ can be joined by a path, and if this path is suitably generic then it gives rise to a cobordism between the corresponding moduli spaces and evaluation maps.  If $X$ is non-compact then we should restrict attention to the sub-class of $J$ for which we can prevent $J$-holomorphic curves escaping to infinity.

\begin{definition}
We define $n_\beta$ to be the degree of $\operatorname{ev} : \mathcal{M}_1(J, \beta) \to L$.  By cobordism-invariance of degree this is independent of $J$, and by Gromov compactness there are only finitely many $\beta$ with $n_\beta \neq 0$.  We extend $n_\beta$ by zero to classes $\beta$ with $\mu(\beta) \neq 2$.
\end{definition}

To obtain integer-valued counts, rather than mod $2$, the moduli space $\mathcal{M}_1(J, \beta)$ needs to be oriented relative to $L$.  It is now well-understood, following de Silva \cite[Theorem Q]{VdS} and Fukaya--Oh--Ohta--Ono \cite[Theorem 8.1.1]{FOOObig}, that such an orientation is provided by a choice of spin structure on $L$.  The torus $L$ carries a \emph{standard} spin structure $\mathfrak{s}$, given by choosing an identification $L \cong \RR^n/\ZZ^n$ and using the trivialisation of $T\RR^n$ to trivialise $TL$, and this is the one we shall always use.

\begin{lemma}
\label{lemDiscCountInvariance}
For $\phi \in \Symp_\infty(X, L)$ we have $n_\beta = n_{\phi_*\beta}$.
\end{lemma}
\begin{proof}
Composition with $\phi$ gives a bijection between discs contributing to $n_\beta$ with respect to $J$ and those contributing to $n_{\phi_*\beta}$ with respect to $\phi_*J$.  All that remains is to show that these discs contribute with the same sign, which amounts to showing that $\phi_* \mathfrak{s} = \mathfrak{s}$.

To do this, recall that a spin structure is a homotopy class of nullhomotopy of the second Stiefel--Whitney class $w_2 : L \to K(\ZZ/2, 2)$.  The Stiefel--Whitney classes can be defined purely homotopy-theoretically using the Wu classes of $L$, so homotopic diffeomorphisms induce the same map on spin structures.  Any diffeomorphism of $L$ is homotopic to a linear one, so it's left to show that linear diffeomorphisms of $L$ preserve $\mathfrak{s}$, and this can be seen directly.
\end{proof}

We thank Oscar Randal-Williams for suggesting the strategy of the second paragraph of this proof.


\subsection{Floer theory of monotone tori}
\label{sscFloer}

Next we review what is known about the Floer theory of $L$.  Again, this is all essentially standard, and follows from constructions and arguments in \cite{BiranCorneaRigidity,BiranCornea} to which more precise references are given at appropriate points.  Fix throughout this subsection a coefficient ring $R$ (associative, commutative, unital).  The case $R = \CC$ is the most often studied, but it will be important for us to consider other choices in order to probe arithmetic features of Floer algebras.

\begin{definition}
The \emph{superpotential} of $L$ is $W \in \ZZ[H_1(L)]$ defined by
\[
W = \sum_\beta n_\beta z^{\partial \beta},
\]
where $z$ is the formal variable whose exponent records the $H_1(L)$ class.  After choosing a basis of loops $\gamma_1, \dots, \gamma_n$ for $H_1(L)$, $W$ can be viewed as a Laurent polynomial in the $z_j \coloneqq z^{\gamma_j}$.  We think of $W$ as an $R$-valued function on the space $H^1(L; R^\times) = \operatorname{Hom}_{\ZZ} (H_1(L), R^\times) \cong (R^\times)^n$ of rank $1$ local systems $\calL$ over $R$ on $L$.  The monomial $z_j$ then records the monodromy of $\calL$ around $\gamma_j$.
\end{definition}

If $\calL_1$ and $\calL_2$ are two such local systems and satisfy $W(\calL_1) = W(\calL_2)$, then we can take the Floer cohomology $HF^*((L, \calL_1), (L, \calL_2))$, which is a $\ZZ/2$-graded $R$-module.  Letting \eqref{eqLcondition} denote the condition
\begin{equation}
\label{eqLcondition}
\tag{$\dag$}
\calL_1 = \calL_2 \text{ and their common value $\calL$ is a critical point of $W$},
\end{equation}
we have the following.

\begin{proposition}
\label{propHFadditive}
If \eqref{eqLcondition} holds then
\begin{equation}
\label{HFadditive}
HF^*((L, \calL_1), (L, \calL_2)) \cong H^*(L; R)
\end{equation}
as $\ZZ/2$-graded $R$-modules.  If \eqref{eqLcondition} fails and $R$ embeds in a finite-dimensional algebra over a field then \eqref{HFadditive} does not hold.
\end{proposition}
\begin{proof}
The ideas are all in \cite[Section 3.3]{BiranCornea}, but we briefly summarise the argument.  The Floer cohomology can be computed using Biran--Cornea's pearl complex $C^*_\mathrm{pearl}$ \cite{BiranCorneaRigidity}.  This starts with the Morse complex $C^*_\mathrm{Morse}(L; \calL_1^\vee \otimes \calL_2)$ associated to a Morse function $f$ on $L$, but then adds index-decreasing corrections to the differential, arising from pseudoholomorphic discs with boundary on $L$.

Since $L$ is a torus, we may take $f$ to be a perfect Morse function.  If $\calL_1 = \calL_2  =\calL$ then the Morse differential vanishes, and the Morse product on $C^*_\mathrm{Morse}(L; \calL_1^\vee \otimes \calL_2) = C^*_\mathrm{Morse}(L; R)$ can be corrected to a pearl product.  Pearl products of the index $1$ critical points span the whole complex, so by the Leibniz rule the pearl differential vanishes if and only if it vanishes on index $1$ critical points.  The latter is equivalent to $\calL$ being a critical point of $W$, so we deduce that \eqref{eqLcondition} implies \eqref{HFadditive}.

Now suppose that $R$ embeds in a finite-dimensional algebra $S$ over a field $\KK$.  If \eqref{eqLcondition} fails then the pearl differential is not identically zero, so
\[
\dim_\KK H^*(C^*_\mathrm{pearl} \otimes_R S) < \dim_\KK (C^*_\mathrm{pearl} \otimes_R S) = \dim_\KK H^*(L; S).
\]
By the universal coefficient theorem, the left-hand side is at least $\dim_\KK (HF^*((L, \calL_1), (L, \calL_2)) \otimes_R S)$.  On the other hand, the right-hand side is $\dim_\KK (H^*(L; R) \otimes S)$.  We conclude that \eqref{HFadditive} does not hold.
\end{proof}

\begin{remark}
If \eqref{eqLcondition} fails and $R$ is itself a field then in fact $HF^*((L, \calL_1), (L, \calL_2))$ vanishes.  We won't use this fact but it can be proved by considering the spectral sequence $C^*_\mathrm{Morse}(L; \calL_1^\vee \otimes \calL_2) \implies HF^*((L, \calL_1), (L, \calL_2))$ associated to the filtration of $C^*_\mathrm{pearl}$ by Morse index (see also \cite{OhSS}).  If $\calL_1 \neq \calL_2$ then the $E_0$-differential kills everything, whilst if $\calL_1 = \calL_2$ is not a critical point of $W$ then the $E_1$-differential kills everything.
\end{remark}

When \eqref{eqLcondition} holds, the isomorphism \eqref{HFadditive} is non-canonical, but there is a canonical inclusion
\begin{equation}
\label{eqPSS}
\operatorname{PSS} : H^{\leq 1}(L; R) \to HF^*((L, \calL), (L, \calL))
\end{equation}
arising from the inclusion of the critical points of index $\leq 1$ into the pearl complex.  There is also a ring structure (associative and unital) on $HF^*$, and $\operatorname{PSS}$ extends to a canonical $\ZZ/2$-graded $R$-algebra isomorphism
\begin{equation}
\label{eqClifford}
\Cl(H^1(L; R), -\tfrac{1}{2}\Hess_\calL W) \cong HF^*((L, \calL), (L, \calL)).
\end{equation}
Again, see \cite[Section 3.3]{BiranCornea} and references therein for proofs of these facts.

The left-hand side of \eqref{eqClifford} requires some explanation.  It denotes the Clifford algebra on $H^1(L; R)$ associated to the quadratic form $-\tfrac{1}{2}\Hess_\calL W$, where $\Hess_\calL W$ is the Hessian of the function $W$ at the point $\calL$.  The Hessian naturally lives on the tangent space $T_\calL H^1(L; R^\times)$, but we identify this with $H^1(L; R)$ by using the translation action of $H^1(L; R^\times)$ on itself to send $\calL$ to the identity, then using the embedding of $H^1(L; R^\times)$ in $H^1(L; R)$.  The coefficients of $\Hess_\calL W$ are always even when $\calL$ is a critical point of $W$, so it makes sense to multiply it by $-\tfrac{1}{2}$ even if $2$ is not invertible in $R$.  Explicitly, suppose we
\begin{equation}
\label{eqCoords}
\begin{aligned}
&\text{fix a basis $\gamma_1, \dots, \gamma_n$ for $H_1(L)$,}
\\ & \text{let $v_1, \dots, v_n$ be the dual basis for $H^1(L)$, and}
\\ & \text{let $z_1, \dots, z_n$ be the associated coordinates on $H^1(L; R^\times)$.}
\end{aligned}
\end{equation}
If $\calL$ has components $\rho_j$ with respect to the coordinates $z_j$ then $v_j$ corresponds to the tangent vector $\rho_j \partial_{z_j}$.  In the Clifford algebra we then have
\begin{equation}
\label{eqCliffordRelations}
v_i^2 = -\frac{\rho_j^2}{2} \left.\frac{\partial^2 W}{\partial z_j^2}\right|_\calL \quad \text{and} \quad v_iv_j + v_jv_i = -\rho_i\rho_j \left.\frac{\partial^2 W}{\partial z_i \partial z_j}\right|_\calL.
\end{equation}
Since we are evaluating at the critical point $\calL$, we can replace each $\rho_j \partial/\partial z_j$ with $z_j \partial/\partial z_j$ if we wish.

\begin{lemma}[{Obvious generalisation of \cite[Proposition 5.4]{Ono}}]
\label{lemLocalSystemAction}
If $\calL$ is a critical point of $W$, and $R$ embeds in a finite-dimensional algebra over a field, then $\HLM$ fixes $\calL$ under its induced action on $H^1(L; R^\times)$.
\end{lemma}
\begin{proof}
The Floer cohomology of two Lagrangians is unchanged if either of the Lagrangians is modified by a Hamiltonian diffeomorphism.  This means, in particular, that if $\phi \in \Ham(X, L)$ then
\begin{equation}
\label{HFcomparison}
HF^*((L, \calL), (L, \phi_*\calL)) \cong HF^*((L, \calL), (L, \calL))
\end{equation}
for all rank $1$ local systems $\calL$.  Comparing with \cref{propHFadditive} we see that if $\calL$ is a critical point of $W$ then $\phi_*$ must fix $\calL$.
\end{proof}

\begin{remark}
\label{rmkBranes}
To define $\ZZ/2$-graded Floer cohomology groups over $\CC$ we really need to keep track of an orientation (equivalent to a $\ZZ/2$-grading in the sense of \cite{SeidelGraded}) and spin structure on each Lagrangian.  We should therefore write $(L, \calL)$ as $(L, o_L, \mathfrak{s}, \calL)$, where $o_L$ is an arbitrary choice of orientation and $\mathfrak{s}$ is the standard spin structure.  The $(L, \phi_*\calL)$ appearing in \eqref{HFcomparison} should then be
\[
(L, \phi_*o_L, \phi_*\mathfrak{s}, \phi_*\calL).
\]
As in \cref{lemDiscCountInvariance}, $\phi_*\mathfrak{s}$ coincides with $\mathfrak{s}$, so we can safely suppress it in our notation.  If $\phi$ is orientation-preserving on $L$ then we can similarly forget the $\phi_*o_L$, but if it's orientation-reversing then the difference between $\phi_*o_L$ and $o_L$ corresponds to a grading shift of $1$.  The upshot is that we can (and will) always implicitly work with $o_L$ and $\mathfrak{s}$, but must remember that if $\phi|_L$ is orientation-reversing then \eqref{HFcomparison} has odd degree.
\end{remark}


\subsection{Proof of \cref{Theorem1}}

Recall \cref{Theorem1} asserts that
\begin{enumerate}
\item\label{Theorem1bitm1} The action of $\HLMS$ on $H_1(L)$ permutes the elements of $B_1$.
\item\label{Theorem1bitm2} If $r = n$ then the induced map $\HLMS \to \Sym(B_1)$ is injective, so $\HLMS$ is finite.
\item\label{Theorem1bitm3} If $r = 0$ then $\HLM$ is trivial.
\end{enumerate}
Here $B_1 \subset H_1(L)$ denotes the set of boundaries of classes $\beta \in H_2(X, L)$ with $n_\beta \neq 0$, and $r$ is the rank of the rational span of $B_1$ in $H_1(L; \QQ)$.

Part \ref{Theorem1bitm1} follows immediately from \cref{lemDiscCountInvariance}.  Part \ref{Theorem1bitm2} is then simply linear algebra: a linear automorphism of $H_1(L)$ is determined by its action on $H_1(L; \QQ)$, which is in turn determined by its action on a spanning set.  For \ref{Theorem1bitm3}, note that if $r=0$ then $W$ is constant so every local system $\calL$ over $\CC$ is a critical point.  \Cref{lemLocalSystemAction} then tells us that $\HLM$ acts trivially on $H^1(L; \CC^*)$ and hence on $H_1(L)$.


\section{Monotone toric fibres}
\label{secToric}


\subsection{Toric geometry background}
\label{sscToricBackground}

In this subsection we briefly review the necessary toric geometry and fix notation.  None of this is original; see for example \cite[Part XI]{CannasDaSilva} and references therein.

Let $\mathfrak{t}$ be the Lie algebra of an abstract $n$-torus $T$.  This contains a lattice $A$, given by the kernel of the exponential map.  We write $\mathfrak{t}^\vee$ for the dual space of $\mathfrak{t}$ and $\langle \cdot, \cdot \rangle$ for the duality pairing.

\begin{definition}
\label{defPolytope}
A \emph{Delzant polytope} is a compact subset $\Delta \subset \mathfrak{t}^\vee$ of the form
\begin{equation}
\label{eqDelta}
\Delta = \{x \in \mathfrak{t}^\vee : \langle x, \nu_j \rangle \geq -\lambda_j \text{ for } j=1, \dots, N\},
\end{equation}
where $\nu_1, \dots, \nu_N$ are elements of $A$ and $\lambda_1, \dots, \lambda_N$ are real numbers, such that:
\begin{itemize}
\item Exactly $k$ facets (codimension-$1$ faces) of $\Delta$ meet at each codimension-$k$ face.
\item Wherever $k$ facets meet, the corresponding normals $\nu_j$ extend to a $\ZZ$-basis for the lattice $A$.
\end{itemize}
We assume that $\Delta$ is non-empty and that none of the inequalities on the right-hand side of \eqref{eqDelta} is redundant, so $\Delta$ has exactly $N$ facets.
\end{definition}

Given a Delzant polytope $\Delta$, we define the associated toric manifold $X$ as follows.  Take $\CC^N$ with coordinates $(w_1, \dots, w_N)$ and symplectic form
\[
\frac{i}{2} \sum_{j=1}^N \diff w_j \wedge \diff \overline{w}_j.
\]
This carries an action of $T^N$ with moment map
\[
\mu_{T^N}(w) = \Big(\frac{1}{2}\lvert w_1 \rvert^2 - \lambda_1, \dots, \frac{1}{2}\lvert w_N \rvert^2 - \lambda_N\Big).
\]
Now let $K \subset \ZZ^N$ be the space of linear relators between the $\nu_j$, i.e.~the kernel of the linear map $\ZZ^N \to A$ sending the $j$th basis vector to $\nu_j$, and let $T_K$ be the subtorus $(K \otimes \RR) / K$ of $T^N = \RR^N/\ZZ^N$.  (This really is a subtorus of $T^N$ since the Delzant condition on $\Delta$ forces the sublattice $K \subset \ZZ^N$ to be primitive.)  This subtorus acts on $\CC^N$, with moment map $\mu_{T_K}$ given by
\[
\mu_{T_K}(w)(\xi) = \xi^T\mu_{T^N}(w)
\]
for $\xi$ in the Lie algebra $K \otimes \RR \subset \RR^N$ of $T_K$.

\begin{definition}
$X$ is the symplectic reduction of $\CC^N$ with respect to this $T_K$-action, at the zero level of the moment map $\mu_{T_K}$.  The conditions on $\Delta$ ensure that the $T_K$-action is free on $\mu_{T_K}^{-1}(0)$, so the reduced space is indeed smooth.
\end{definition}

The $T^N$-action on $\CC^N$ descends to a Hamiltonian action of $T^N / T_K$ on $X$.  Note that $\ZZ^N / K$ is canonically identified with $A$, so $T^N / T_K$ is identified with our original abstract torus $T$, whose Lie algebra is $\mathfrak{t}$.  Thus $X$ carries a Hamiltonian $T$-action, and the image of its moment map $\mu_T : X \to \mathfrak{t}^\vee$ is exactly $\Delta$.  We write $F_j$ for the $j$th facet of $\Delta$, obtained by replacing the $j$th inequality of \eqref{eqDelta} with an equality.  The preimages $D_j$ of the $F_j$ under $\mu_T$ are the \emph{toric divisors}.

\begin{definition}
A \emph{Lagrangian toric fibre} is a fibre $L = \mu_T^{-1}(p)$ of the moment map over a point $p$ in the interior of $\Delta$.  It is a free Lagrangian $T$-orbit.  By translating $\Delta$ we may assume that $p = 0$, and then $L$ is monotone if and only if all $\lambda_j$ are equal to some common value $\lambda$.  We will assume that this holds from now on.
\end{definition}

Since $L$ is a free $T$-orbit, we have identifications
\[
H_1(L) = H_1(T) = \ker (\exp : \mathfrak{t} \to T) = A.
\]
The toric divisors $D_1, \dots, D_N$ form a basis for $H_{2n-2}(X, L)$, and the dual basis elements for $H_2(X, L)$ are called \emph{basic classes} and denoted $\beta_1, \dots, \beta_N$.  The boundary of $\beta_j$ is identified with $\nu_j \in A$.

Each $\beta_j$ is represented by a holomorphic disc $u_j$ with respect to the standard complex structure, that intersects $D_j$ once, positively, and is disjoint from the other $D_l$.  Cho and Cho--Oh \cite{Cho, ChoOh} showed that the $u_j$ are the only holomorphic index $2$ discs with boundary on $L$, up to translation by $T$.  They also showed that these discs are regular (i.e.~the standard complex structure is regular in the sense of \cref{sscDiscCounts}) and count positively with respect to the standard spin structure.  So we get
\begin{equation}
\label{eqToricnbeta}
n_\beta = \begin{cases} 1 & \text{if $\beta = \beta_j$ for some $j$} \\ 0 & \text{otherwise}.\end{cases}
\end{equation}
Hence $B_1 = \{\nu_1, \dots, \nu_N\}$, which spans $H_1(L)$.  Note that the $\nu_j$ are all distinct, so it is unambiguous to talk about how $\HLM$ or $\HLMS$ permutes the $\nu_j$.


\subsection{Proof of \cref{Theorem2} for compact $X$}
\label{sscToricProof}

By \cref{Theorem1} we can view $\HLM$ and $\HLMS$ as subgroups of the symmetric group $S_N$ on the $\nu_j$, and the task is to show that they are precisely the subgroups that fix $K$ pointwise and setwise respectively.

The groups $\Ham(X, L)$ and $\Symp_\infty(X, L)$ naturally act linearly on $H_2(X, L)$, where they permute the $\beta_j$ by \cref{lemDiscCountInvariance} and \eqref{eqToricnbeta}.  The map $H_{>0}(L) \to H_{>0}(X)$ vanishes, so we get a short exact sequence
\[
0 \to H_2(X) \to H_2(X, L) \xrightarrow{\ \partial\ } H_1(L) \to 0,
\]
under which the linear action on $H_2(X, L)$ and permutation action on the $\beta_j$ induce the linear action on $H_1(L)$ and permutation action on the $\nu_j$ that we care about.  Since the $\beta_j$ form a basis for $H_2(X, L)$, the relators $K$ between the $\nu_j$ correspond precisely to the kernel of the boundary map $H_2(X, L) \to H_1(L)$, and hence to the image of $H_2(X)$ in $H_2(X, L)$.  Another consequence of the $\beta_j$ forming a basis is that every permutation of them corresponds to a unique linear endomorphism of $H_2(X, L)$, so it makes sense to talk about how such a permutation acts on $H_2(X, L)$.  In light of these observations, \cref{Theorem2} is reduced to showing that the action of $\HLMS$ on the $\beta_j$ realises precisely those permutations fixing the image of $H_2(X)$ in $H_2(X, L)$ setwise, whilst the action of $\HLM$ realises precisely those permutations fixing the image of $H_2(X)$ pointwise.

To see that all permutations arising from $\HLMS$ fix the image of $H_2(X)$ setwise simply note that the short exact sequence above is equivariant with respect to the action of $\Symp_\infty(X, L)$.  To see moreover that those arising from $\HLM$ fix the image of $H_2(X)$ pointwise, observe that every $\phi \in \Ham(X, L)$ is isotopic to the identity as a diffeomorphism of $X$, so acts trivially on $H_*(X)$.  It therefore remains to show that all of the permutations allowed by these constraints can be realised.

First we deal with $\HLMS$, so let $\sigma \in S_N$ be a permutation fixing $K$ setwise.  It induces a linear automorphism $\theta$ of $K$, and hence an automorphism $f$ of the torus $T_K$.  It also acts on $\CC^N$ by the symplectomorphism $\phi$ given by permuting the factors.  Note that $\phi$ preserves the torus $L_{\CC^n} = \mu_{T^N}^{-1}(0)$ setwise, and acts on
\[
H_1(L) = H_1(L_{\CC^n}/T_K) = \Big(\bigoplus_{j=1}^N \ZZ \nu_j\Big) / K
\]
by permuting the $\nu_j$ according to $\sigma$.  It thus suffices to show that $\phi$ descends to a symplectomorphism of $X$, which is a consequence of the following two facts:
\begin{itemize}
\item $\phi$ preserves $\mu_{T_K}^{-1}(0)$ setwise, because if $w$ is a point in this set then for all $\xi \in K$ we have
\[
\mu_{T_K}(\phi(w))(\xi) = \xi^T\mu_{T^N}(\phi(w)) = \theta^{-1}(\xi)^T \mu_{T^N}(w) = \mu_{T_K}(w)(\theta^{-1}(\xi)) = 0.
\]
\item Quotienting this level set by $T_K$ commutes with $\phi$, because for all $w \in \CC^N$ and all $g \in T_K$ we have $\phi(g\cdot w) = f(g)\cdot \phi(w)$.
\end{itemize}

Now we deal with $\HLM$, so let $\sigma \in S_N$ be a permutation fixing $K$ pointwise, and let $\phi$ denote the action of $\sigma$ on $\CC^N$ by permuting the factors.  This $\phi$ induces the desired action on $H_1(L)$, as in the case of $\HLMS$, so it suffices to show that $\phi$ descends to a Hamiltonian diffeomorphism of $X$.  To do this, we'll construct a Hamiltonian action of a connected Lie group $G$ on $\CC^N$ such that $\phi$ can be realised by an element of $G$, and then show that this $G$-action descends to a Hamiltonian action on $X$.

To construct the action on $\CC^N$, first we partition the $\nu_j$ by the equivalence relation that says $\nu_i \sim \nu_j$ if and only if $\nu_i$ and $\nu_j$ appear with the same coefficient in every element of $K$.  Then $\sigma$ preserving $K$ pointwise is equivalent to it only permuting the $\nu_j$ within their parts of the partition.  By reordering the $\nu_j$, we may assume that the partition is
\[
\{\nu_1, \dots, \nu_{N_1}\}, \{\nu_{N_1+1}, \dots, \nu_{N_1+N_2}\}, \dots, \{\nu_{N-N_k+1}, \dots, \nu_N\},\]
so $\sigma$ lies in $S_{N_1} \times S_{N_2} \times \dots \times S_{N_k}$.

\begin{remark}
\label{rmkConjectureToric}
The number of pieces $k$ of the partition is at least the rank of $K$, which is $N-n$, so we have
\[
\sum_{j=1}^k (N_j - 1) = N - k \leq n.
\]
So case \ref{conjitm2} of \cref{conjHighDim} holds in this case.
\end{remark}

Let $G = \U(N_1) \times \dots \times \U(N_k)$ be the block-diagonal subgroup of $\U(N)$, acting on $\CC^N$ in the obvious way. Since $\sigma$ preserves the partition, $\phi$ can be realised by an element of $G$, with each block a permutation matrix.  Moreover, this $G$-action is Hamiltonian, with moment map
\[
\mu_G(w)(\xi) = \frac{i}{2}w^\dag \xi w
\]
for all $w \in \CC^N$ and all $\xi \in \mathfrak{g}$.  Here we are thinking of $\xi$ as a skew-Hermitian block-diagonal matrix, and ${}^\dag$ denotes conjugate transpose.

It now remains to show that this $G$-action descends to a Hamiltonian action on $X$.  For this, note that the action of $T_K$ on $\CC^N$ is by block-diagonal (in fact, diagonal) unitary matrices, with the same block sizes as $G$, and by definition of the partition each block is a scalar matrix.  Thus the action of $T_K$ commutes with every $\xi$ in $\mathfrak{g}$, and hence preserves $\mu_G$.  This forces the $G$- and $T_K$-actions to commute, and forces $G$ to preserve $\mu_{T_K}$---this is essentially Noether's theorem---which means that the $G$-action descends as claimed.


\subsection{Extension to non-compact $X$}
\label{sscNonCompactToric}

Now suppose we remove the compactness condition on $\Delta$ in \cref{defPolytope}.  Instead we ask that $\Delta$ has at least one vertex.

\begin{remark}
This does not produce all non-compact toric varieties, only those that can be constructed as symplectic reductions (or GIT quotients) of affine space.  In particular, it excludes things like $\CC^*$, where the presence of non-trivial $H_1(X)$ breaks our description of $H_2(X, L)$, and $\CC\PP^2$ minus a toric fixed point, where holomorphic curves can escape into the deleted point.
\end{remark}

The construction of $X$, and permutation descriptions of $\HLM$ and $\HLMS$, then go through as above except for the following modification.  $X$ carries a standard complex structure $J_0$, and we restrict attention to almost complex structures $J$ that agree with $J_0$ outside a compact set.  We then take $\Symp_\infty(X, L)$ to be those symplectomorphisms $\phi$ that preserve $L$ setwise and are $J_0$-holomorphic outside a compact set, since these $\phi$ preserve the class of allowed $J$.  Note that the symplectomorphisms of $X$ induced by permutations of the factors in $\CC^N$ have this property, so the construction of elements in $\HLMS$ in the previous subsection remains valid.  We can also use our earlier construction of elements in $\HLM$, and simply cut off the generating Hamiltonians outside a large compact set containing the sweepout of $L$.

It remains to explain why a generic choice of such $J$ is regular in the sense of \cref{sscDiscCounts}, and why using such $J$ prevents pseudoholomorphic curves from escaping to infinity.  The former holds since within this class of almost complex structures we are free to perturb $J$ on a neighbourhood of $L$, through which any disc with boundary on $L$ must pass.  For the latter, meanwhile, note that $X$ has an affinisation $X_{\mathrm{aff}}$, given by the spectrum of the ring of global ($J_0$-)holomorphic functions on $X$.  This comes with a ($J_0$-)holomorphic map $\pi : X \to X_{\mathrm{aff}}$, and in our toric setting this map is projective---see \cite[Lemma 4.2]{SmithQH} for example.  Using this, we can prove the following.

\begin{lemma}
For any compact set $V \subset X$ there exists a compact set $W \subset X$ such that: for any almost complex structure $J$ on $X$ that agrees with $J_0$ outside $V$, any $J$-holomorphic disc $u : (D^2, \partial D^2) \to (X, L)$ is contained in $W$.
\end{lemma}
\begin{proof}
Fix an embedding of $X_{\mathrm{aff}}$ in an affine space $\CC^m$, and let $z_1, \dots, z_m$ be the corresponding coordinate functions.  Let $r$ be the maximum of $\lvert z_1 \rvert, \dots, \lvert z_m \rvert$ over $\pi(V \cup L)$, and define
\[
W = \pi^{-1}(\{z \in X_{\mathrm{aff}} : \lvert z_j \rvert \leq r \text{ for all $j$}\}).
\]
This is compact since $\pi$ is projective and hence proper.  Note also that $L \subset \pi(W)$.

Now suppose $u$ is a $J$-holomorphic disc that escapes $W$---say the $j$th component of $\pi \circ u$ exceeds $r$ somewhere.  Let $f$ denote $z_j \circ \pi \circ u$, and let $p \in D^2$ be a point at which $\lvert f \rvert$ attains its maximum.  Then $f$ is holomorphic, in the ordinary sense, on a neighbourhood of $p$, since $J = J_0$ near $u(p)$.  However, $f$ is not open near $p$ since it lands in the closed disc of radius $\lvert f(p) \rvert$.  This forces $f$ to be constant, but this is impossible since $u(\partial D^2) \subset L \subset W$.  So no such $u$ can exist.
\end{proof}


\section{Further constraints}
\label{secFurtherConstraints}


\subsection{Continuation elements}

In this subsection we return to studying the Floer theory of $L$ and derive an additional constraint on $\HLM$ that will be used later.

A compactly supported symplectomorphism $\phi$ of $X$ induces a pushforward homomorphism
\[
\Phi : HF^*(L_1, L_2) \to HF^*(\phi(L_1), \phi(L_2))
\]
for any Lagrangians $L_1$ and $L_2$ for which Floer cohomology can be defined.  Assume that there exists a compactly supported Hamiltonian isotopy $\phi_t$ from the identity to $\phi$.  This induces a continuation element $c_j \in HF^0(L_j, \phi(L_j))$ for each $j$, such that for $a \in HF^*(L_1, L_2)$ we have
\begin{equation}
\label{eqContinuation}
\Phi(a) c_1 = c_2 a.
\end{equation}
We are following Seidel's sign conventions from \cite[Equation (1.8)]{SeidelBook} (with $d=1$, $\mathscr{F}_0^1 = \id$, $\mathscr{F}_1^1 = \Phi$, and $T^0 = c_\bullet$) and \cite[Equation (1.3)]{SeidelBook}; see also \cite[Section (10c)]{SeidelBook}.  We are also suppressing orientations, spin structures, and local systems in our notation.

Now consider the case where $L_1$ and $L_2$ are both equal to our monotone torus $L$, equipped with an arbitrary orientation, the standard spin structure $\mathfrak{s}$, and a local system $\calL$ over a ring $R$ that is a critical point of the superpotential $W$.  Suppose that $\phi$ preserves $L$ setwise.  Then by \cref{lemLocalSystemAction} and \cref{rmkBranes} we can identify $\phi(L, \calL)$ with $(L[d], \calL)$, where $[d]$ denotes a grading shift of $d$ mod $2$ and $d=0$ or $1$ if $\phi|_L$ is orientation-preserving or -reversing respectively.  We can thus think of $c_1 = c_2$ as an element $c$ of $HF^d((L, \calL), (L, \calL))$, and consider the composition
\[
\Phi' : HF^*((L, \calL), (L, \calL)) \xrightarrow{\ \Phi\ } HF^*((L[d], \calL), (L[d], \calL)) \xrightarrow{\ S\ } HF^*((L, \calL), (L, \calL)),
\]
where $S$ is the shift isomorphism $a \mapsto (-1)^{d\lvert a \rvert}a$.  This map $\Phi'$ is intertwined with the classical pushforward $\phi_* = (\phi^*)^{-1} : H^*(L; R) \to H^*(L; R)$ by the PSS map \eqref{eqPSS}.  The element $c$ is invertible---its inverse is the continuation element associated to the reverse Hamiltonian isotopy---so \eqref{eqContinuation} becomes
\[
\Phi'(a) = (-1)^{d\lvert a \rvert} c a c^{-1}.
\]

The upshot of this discussion is the following.

\begin{lemma}
\label{lemConjugation}
For each $g \in \HLM$, and each critical point $\calL$ of $W$ over $R$, the dual action of $g$ on $H^1(L; R)$ corresponds under the PSS map to $(-1)^d$ times conjugation by an element $c \in HF^d((L, \calL), (L, \calL))$.  Here $d$ is $0$ if $g$ is orientation-preserving and $1$ otherwise.\hfill$\qed$
\end{lemma}

A version of this idea was used by Varolgunes in \cite{Varolgunes} to constrain Hamiltonian actions on a Lagrangian nodal sphere, and we are grateful to him for explaining it to us.

\begin{remark}
The argument of Hu--Lalonde--Leclercq \cite{HuLalondeLeclercq} can be phrased in this language, as follows.  For weakly exact $L$ the PSS map extends to a canonical ring isomorphism $H^*(L) \to HF^*(L, L)$, and hence:
\begin{itemize}
\item The action of $g \in \HLM$ on $H^*(L)$ is determined by its action on $HF^*(L, L)$.
\item The action on $HF^*(L, L)$ is trivial since it is given by $a \mapsto (-1)^{d\lvert a \rvert} c a c^{-1}$ for some $c$, but the ring is graded commutative.
\end{itemize}
(When $L$ is not a torus, one has to be careful about how $\HLM$ acts on spin structures, or work in characteristic $2$ so that this is irrelevant.)  It was pointed out to us by Jake Solomon that this argument could be extended to other settings where $HF^*(L, L)$ is known to be graded commutative, e.g.~\cite{FOOOinvolution,Pascaleff,Solomon}, and where the PSS map allows the action on $H^*(L)$ to be deduced from the action on $HF^*(L, L)$.
\end{remark}


\subsection{$1$-eigenspaces in $\HLM$}

The other constraint that we introduce is based on a method for detecting critical points of $W$.  So fix a coefficient ring $R$, choose bases and coordinates as in \eqref{eqCoords}, and identify tangent spaces of $H^1(L; R^\times)$ with $H^1(L; R)$ as described in the same place.  Given $g \in \HLM$ let $\psi_g$ denote the automorphism of $H^1(L; R^\times)$ it induces, defined by $z^\gamma(\psi_g(\calL)) = z^{g(\gamma)}(\calL)$ for all monomials $z^\gamma$ and all local systems $\calL$.  Under our tangent space identifications the derivative of $\psi_g$ becomes the dual map $g^\vee : H^1(L; R) \to H^1(L; R)$.

\begin{lemma}
\label{lemKerIntersect}
If $g_1, \dots, g_k$ are elements of $\HLM$ such that
\begin{equation}
\label{eqKerIntersect}
\bigcap_{j=1}^k \ker_R(g_j - I) = 0,
\end{equation}
where $\ker_R(g_j - I)$ denotes the kernel of $g_j - I$ acting on $H_1(L; R)$, then any simultaneous fixed point $\calL \in H^1(L; R^\times)$ of the $\psi_{g_j}$ is a critical point of the superpotential $W$.
\end{lemma}
\begin{proof}
Suppose $\calL$ is a simultaneous fixed point.  By $\HLM$-invariance of the superpotential $W$ we have $W \circ \psi_{g_j} = W$ for all $j$.  Differentiating, and using our tangent space identifications, we get
\[
\diff_\calL W \circ g_j^\vee = \diff_\calL W
\]
for all $j$.  This equality can be rewritten as $(g_j^{\vee\vee}-I)\diff_\calL W = 0$, so under the double-duality isomorphism $\diff_\calL W$ lies in the intersection of the kernels of the $g_j - I$.  If \eqref{eqKerIntersect} holds then we conclude that $\diff_\calL W = 0$, i.e.~$\calL$ is a critical point of $W$.
\end{proof}

\begin{example}
\label{exMinusId}
If $-I$ is in $\HLM$ then, with respect to any basis of $H_1(L)$, every local system in $\{\pm 1\}^n$ is a critical point of $W$.
\end{example}

\Cref{lemLocalSystemAction} then gives the following.

\begin{corollary}
\label{corEigenspaces}
If $g_1, \dots, g_k \in \HLM$ satisfy \eqref{eqKerIntersect}, and if $R$ embeds in a finite-dimensional algebra over a field, then every $\calL$ fixed by the $\psi_{g_j}$ is fixed by $\psi_g$ for all $g \in \HLM$.\hfill$\qed$
\end{corollary}


\section{Classification for $n=2$}
\label{secn2}

Throughout this section we assume that the dimension $n$ of our monotone torus $L \subset X$ is $2$.  We also assume that the rank $r$ of the rational span of $B_1 = \{\partial \beta : n_\beta \neq 0\} \subset H_1(L)$ is greater than $0$.  Our main goal is to prove \cref{Theorem3,Theorem4}, which give strong constraints on $\HLM$ in this dimension.  The section ends with a short discussion of a different but closely related question: given $\HLM$, what can be said about the Floer theory of $L$?


\subsection{Proof of \cref{Theorem3}}

Suppose that $X$ is symplectically aspherical ($\omega$ vanishes on $\pi_2(X)$) and simply connected.  The latter ensures that any loop $\gamma \in \pi_1(L) = H_1(L)$ can be capped off by a disc $\beta \in \pi_2(X, L)$, and the former then tells us that any two cappings have the same area.  It thus makes sense to talk about the area, or equivalently (by monotonicity) the Maslov index, of any capping of a loop $\gamma$.  Let $\overline{\mu} : H_1(L) \to \ZZ$ denote this induced Maslov index homomorphism.  The set $\overline{\mu}^{-1}(2)$ is non-empty, since it contains $B_1$, and is a torsor for $\ker \overline{\mu} \cong \ZZ$.  \cref{Theorem3} states that $\HLMS$ naturally embeds in the infinite dihedral group on this set, and that the image is either trivial or generated by a single involution.

The first part of the statement follows immediately from the fact that $\HLMS$ preserves $\overline{\mu}$.  To prove the second part, note that by \cref{lemDiscCountInvariance} the action of $\HLMS$ on $\overline{\mu}^{-1}(2)$ preserves $B_1$ setwise.  Since $B_1$ is finite, by Gromov compactness, this prevents the action from containing any translations of $\overline{\mu}^{-1}(2)$.  The action is therefore either trivial or generated by a single involution.

\begin{example}
The hypotheses are satisfied for the monotone Lagrangian \emph{matching torus} $\mathbb{T}_p$ in the $A_{p-1}$ Milnor fibre, constructed by Lekili--Maydanskiy in \cite{LekiliMaydanskiy}.  Here $p$ is any positive integer, and Lekili and Maydanskiy show that $\mathbb{T}_1$ and $\mathbb{T}_2$ recover the Clifford torus in $\CC^2$ and the Polterovich torus in $T^*S^2$ \cite{AlbersFrauenfelder} respectively.  They also construct, for coprime positive integers $p$ and $q$ with $p > q$, a Stein surface $B_{p, q}$ containing a monotone Lagrangian torus $\mathbb{T}_{p, q}$.  These arise from quotienting the $A_{p-1}$ Milnor fibre and the torus $\mathbb{T}_p$ by a certain action of $\ZZ/p$, which depends on $q$.  Topologically $B_{p,q}$ is a rational homology ball.  The hypotheses of \cref{Theorem3} do not apply directly to $\mathbb{T}_{p,q}$, since $B_{p,q}$ is not simply connected.  But a straightforward modification of the argument does apply: given a class $\gamma$ in $\pi_1(\mathbb{T}_{p,q})$, the multiple $p\gamma$ can be capped off in $B_{p,q}$; any two cappings differ by an element of $H_2(B_{p,q})$, which necessarily has area zero since this group is torsion (as $B_{p,q}$ is a rational homology ball); so every $\gamma$ has a well-defined `capping area', and the proof now goes through as before.
\end{example}


\subsection{Proof of \cref{Theorem4}\ref{rk1}}
\label{sscrk1proof}

Now focus on the case $n=2$, $r=1$.  We wish to show that for a suitable choice of basis for $H_1(L)$ we have either:
\begin{itemize}
\item $W = a\pm x$ for some $a \in \ZZ$, and $\HLM$ is a subgroup of
\begin{equation}
\label{eqFirstBP}
\left\{ \begin{pmatrix} 1 & \ZZ \\ 0 & \pm1 \end{pmatrix} \right\} \subset \GL(2, \ZZ).
\end{equation}
\item $W = a \pm \big(x + \frac{1}{x}\big)$ for some $a \in \ZZ$, and $\HLM$ is a subgroup of
\begin{equation}
\label{eqSecondBP}
\left\{ \begin{pmatrix} \pm 1 & 2\ZZ \\ 0 & 1 \end{pmatrix} \right\} \subset \GL(2, \ZZ).
\end{equation}
\item $W = a \pm \big(x - \frac{1}{x}\big)$ for some $a \in \ZZ$, and $\HLM$ is a subgroup of
\[
\left\{ \begin{pmatrix} 1 & 4\ZZ \\ 0 & 1 \end{pmatrix} \right\} \subset \GL(2, \ZZ).
\]
\item $\HLM$ is trivial.
\end{itemize}


To begin, choose a basis $\gamma_1$, $\gamma_2$ for $H_1(L)$ such that $B_1$ lies in the span of $\gamma_1$.  Since $\HLM$ preserves $B_1$, we see that $\gamma_1$ must be a common eigenvector for $\HLM$, so $\HLM$ is a subgroup of
\[
\left\{ \begin{pmatrix} \pm 1 & \ZZ \\ 0 & \pm 1 \end{pmatrix} \right\} \subset \GL(2, \ZZ).
\]
Letting $x$ and $y$ be the corresponding monomials (i.e.~$z^{\gamma_1}$ and $z^{\gamma_2}$) in $\ZZ[H_1(L)]$, we also see that the superpotential $W$ is a function of $x$ only.  We now split into several cases: if $W$ doesn't have a critical point over $\CC$ then we'll see that the first or fourth bullet point holds, depending on whether $W$ has a critical point in positive characteristic; if $W$ has a critical point over $\CC$ and $\HLM$ contains a non-trivial element of positive determinant then we'll see that the second or third bullet point holds; and finally if $W$ has a critical point over $\CC$ and contains an element of negative determinant then we'll see that the second bullet point holds.

First suppose that $W$ has no critical points over $\CC$.  Then $\frac{\partial W}{\partial x}$ must be a non-zero multiple of a single monomial.  Since $W$ is an integer Laurent polynomial in $x$, we must have $W = a+bx^k$ for some $a, b, k \in \ZZ$ with $b, k \neq 0$.  By replacing $\gamma_1$ with $-\gamma_1$ if necessary, we may assume that $k > 0$.  Since $W$ is invariant under $\HLM$, we see that $\HLM$ is a subgroup of \eqref{eqFirstBP}.  To complete this case it suffices to show that either $k=1$ and $b = \pm 1$, or that $\HLM$ is trivial, so that the first or fourth bullet point holds.  Suppose then that $k > 1$ or $b \neq \pm 1$, and let $\KK$ be an infinite field of characteristic dividing $kb$.  Then $\diff W$ is identically zero over $\KK$, so every point of $H^1(L; \KK^\times)$ is a critical point and hence must be fixed by $\HLM$, by \cref{lemLocalSystemAction}.  Because $\KK$ is infinite, this forces $\HLM$ to be trivial.

Now suppose that $W$ does have a critical point over $\CC$, say $(\xi, \eta_0)$.  Since $W$ is independent of $y$, we see that $(\xi, \eta)$ is a critical point for all $\eta \in \CC^*$, and since $\HLM$ fixes critical points (\cref{lemLocalSystemAction} again) we get
\begin{equation}
\label{eqShear}
\xi^{\eps_1} = \xi \quad \text{and} \quad \xi^m\eta^{\eps_2} = \eta \quad \text{for all} \quad \begin{pmatrix} \eps_1 & m \\ 0 & \eps_2 \end{pmatrix} \in \HLM \quad \text{and all} \quad \eta \in \CC^*.
\end{equation}
The second equality tells us that $\eps_2$ must be $1$.

Suppose next that $\HLM$ contains
\[
g = \begin{pmatrix} 1 & m \\ 0 & 1 \end{pmatrix}
\]
for some $m \neq 0$.  From \eqref{eqShear} we have that $\xi$ is an $m$th root of unity, and we may assume that $\eta = 1$.  Equip $L$ with the local system $\calL = (\xi, 1)$ over $\ZZ[\xi]$, and let $HF^*$ be shorthand for $HF^*((L, \calL), (L, \calL))$.  Let $u$ and $v$ be the basis for $H^1(L)$ dual to our basis for $H_1(L)$, viewed as elements of $HF^1$ via the PSS map.  By \eqref{eqClifford} and \eqref{eqCliffordRelations}, $HF^*$ is a Clifford algebra on $u$ and $v$, in which $u^2 = \lambda$, $uv+vu = \mu$, and $v^2 = \nu$, where $\lambda$, $\mu$, and $\nu$ are elements of $\ZZ[\xi]$ given by
\[
\lambda = -\frac{\xi^2}{2}\frac{\partial^2 W}{\partial x^2}(\xi ,1) \text{, \quad} \mu = -\xi\frac{\partial^2 W}{\partial x \, \partial y}(\xi,1) \text{, \quad and \quad} \nu = -\frac{1}{2}\frac{\partial^2 W}{\partial y^2}(\xi,1).
\]
In our case we immediately get $\mu =\nu = 0$.

Consider the action of $g^\vee$ on $H^1(L)$.  This sends $u$ to $u+mv$ and $v$ to $v$.  By \cref{lemConjugation} there exists $c \in HF^0$, defined over $\ZZ[\xi]$, with inverse defined over the same ring, such that
\[
cu = (u+mv)c \quad \text{and} \quad cv = vc.
\]
Writing $c = p+quv$ and $c^{-1} = s + tuv$, with $p,q, r, s \in \ZZ[\xi]$, the equations reduce to $ps = 1$, $pt+qs=0$, and $2q\lambda = mp$.  In particular, we have $m = 2qs\lambda$, so $\lambda \neq 0$ and hence $\xi$ is not a repeated root of $\frac{\partial W}{\partial x}$.

Therefore every root of $\frac{\partial W}{\partial x}$ over $\CC$ is an $m$th root of unity, and none of these is a repeated root.  Since $\frac{\partial W}{\partial x}$ has integer coefficients, it must be of the form
\begin{equation}
\label{eqProdCyclo}
lx^k \prod_j \Phi_{d_j}(x),
\end{equation}
where $\Phi_d$ denotes the $d$th cyclotomic polynomial, the $d_j$ are distinct factors of $m$, and $l$ and $k$ are integers with $l \neq 0$.  If $l \neq \pm 1$ then by working in an infinite field $\KK$ of characteristic dividing $l$ we see that $\HLM$ is trivial, as above, contradicting the existence of $g$.  Hence $l = \pm 1$.

Now we can use the fact that $W$ itself has integer coefficients.  By considering the highest and lowest degree terms in \eqref{eqProdCyclo}, we see that $k+1$ and $k+1+\sum_j \phi(d_j)$ are both $\pm 1$, where $\phi(d_j)$ is the degree of $\Phi_{d_j}$.  The only possibility is that $k = -2$ and the $d_j$ are $\{1, 2\}$, $\{3\}$, $\{4\}$, or $\{6\}$.  This results in the following possibilities for $\frac{\partial W}{\partial x}$, up to an overall sign:
\[
1- \frac{1}{x^2}, \quad 1+\frac{1}{x}+\frac{1}{x^2}, \quad 1+\frac{1}{x^2}, \text{\quad and \quad} 1-\frac{1}{x}+\frac{1}{x^2}.
\]
The second and fourth options cannot be the derivative of a Laurent polynomial, so we are left with the first and third, corresponding to $W = a \pm (x+\frac{1}{x})$ and $W = a \pm (x - \frac{1}{x})$ for some $a \in \ZZ$ respectively.  If $W = a \pm (x+\frac{1}{x})$ then the critical locus is $\{\pm 1\} \times \CC^*$ and \eqref{eqShear} tells us that $\HLM$ is a subgroup of \eqref{eqSecondBP}, so the second bullet point holds.  If $W = a \pm (x-\frac{1}{x})$ instead then the critical locus is $\{\pm i\} \times \CC^*$ and we similarly see that the third bullet point holds.

Finally suppose $\HLM$ contains
\[
g = \begin{pmatrix} -1 & m \\ 0 & 1 \end{pmatrix}
\]
for some $m \in \ZZ$.  Since $W$ is invariant under $g$, it must be of the form
\[
n_0 + \sum_{j > 0} n_j\Big(x^j+\frac{1}{x^j}\Big)
\]
for some $n_j \in \ZZ$.  This forces the critical locus of $W$ to contain $\{\pm 1\} \times \CC^*$, and in fact this containment must be an equality since $g$ fixes critical points.  Equipping $L$ with the local system $\calL = (\xi = \pm 1, 1)$ over $\ZZ$, and considering the action of $g$ on $HF^*$, we get a continuation element $c \in HF^1$.  The existence of an invertible element in $HF^1$ forces $\Hess_\calL W$ to be non-zero, which means that $\xi$ cannot be a repeated root of $\frac{\partial W}{\partial x}$.  We are now back in the situation of \eqref{eqProdCyclo}, with $m=2$, and by the above argument we conclude that that the second bullet point holds.


\subsection{Proof of \cref{Theorem4}\ref{rk2}}

Suppose instead that $r=2$.  Then $\HLM$ is finite, by \cref{Theorem1}\ref{Theorem1itm2}, so is in one of the $13$ conjugacy classes of finite subgroups of $\GL(2, \ZZ)$.  These can be understood as follows.

The orientation-preserving part $\HLM^+$ of $\HLM$ is cyclic, and is determined up to conjugacy by its order, which may be $1$, $2$, $3$, $4$, or $6$.  We denote the conjugacy class of the group of order $m$ by $\gp{m}$.  The full group $\HLM$ is then either $\HLM^+$ itself, or is an extension of $\ZZ/2$ by $\HLM^+$, generated by $\HLM^+$ and a single orientation-reversing element $g$.

If $\HLM^+$ has order $1$ or $2$ then $\HLM$ is completely determined, up to conjugacy, by the conjugacy class of this element $g$.  Any such $g$ (orientation-reversing and of finite order in $\GL(2, \ZZ)$) is conjugate to exactly one of
\[
\begin{pmatrix} 1 & 0 \\ 0 & -1 \end{pmatrix} \quad \text{and} \quad \begin{pmatrix} 0 & 1 \\ 1 & 0 \end{pmatrix}.
\]
We denote these by $g_\rmf$ and $g_\rmt$ respectively, since they \emph{fix} a basis vector or \emph{transpose} two basis vectors.  We write $\gp{1}_\rmf$, $\gp{1}_\rmt$, $\gp{2}_\rmf$, and $\gp{2}_\rmt$ for the corresponding extensions of $\ZZ/2$ by $\gp{1}$ and $\gp{2}$.

If $\HLM^+$ has order $3$, $4$, or $6$, then there is a unique $\HLM$-invariant inner product on $H_1(L; \RR)$ such that the shortest non-zero element of $H_1(L)$ has length $1$.  The intersection of $H_1(L)$ with the unit circle is then a regular hexagon, square, or regular hexagon respectively, and we can identify $\HLM$ with a conjugacy class of subgroup of the dihedral symmetry group of the corresponding polygon.  For $\lvert\HLM^+\rvert = 4$, there are two possibilities: $\HLM = \HLM^+ = \gp{4}$; or $\HLM$ is the full dihedral group, which we denote by $\gp{4}_\rmd$.  Similarly for $\lvert\HLM^+\rvert = 6$ we have $\HLM = \gp{6}$ or $\gp{6}_\rmd$.  For $\lvert\HLM^+\rvert = 3$, meanwhile, there are three possibilities: $\HLM = \gp{3}$; $\HLM = \gp{3}_\rmv$, generated by $\gp{3}$ and a reflection in an axis through a vertex of the hexagon; or $\HLM = \gp{3}_\rme$, generated by $\gp{3}$ and a reflection in an axis through the midpoint of an edge.

The assertion of \cref{Theorem4}\ref{rk2} is made precise by the following.

\begin{proposition}
\label{propMainClassification}
The groups $\gp{2}_\rmt$, $\gp{3}_\rme$, $\gp{4}$, $\gp{4}_\rmd$, $\gp{6}$, and $\gp{6}_\rmd$ cannot occur as $\HLM$.  The groups $\gp{1}$, $\gp{1}_\rmf$, $\gp{1}_\rmt$, $\gp{2}_\rmf$, and $\gp{3}_\rmv$ can occur as $\HLM$ for monotone toric fibres $L$.  The groups $\gp{2}$ and $\gp{3}$ cannot occur as $\HLM$ for monotone toric fibres but we are unable to rule them out in general.
\end{proposition}

Before proving this we establish a simple connection between $\HLM^+$ and critical points of $W$.

\begin{lemma}
\label{lemExistenceOfCriticalPoints}
If $\lvert \HLM^+ \rvert$ is even then, with respect to any basis of $H_1(L)$, the critical locus of $W$ over $\CC$ is $\{\pm 1\} \times \{\pm 1\}$.  If $\lvert \HLM^+\rvert$ is divisible by $3$, and we fix a basis of $H_1(L)$ so that the order $3$ elements are
\begin{equation}
\label{eqOrder3}
\begin{pmatrix} 0 & -1 \\ 1 & -1 \end{pmatrix} \quad \text{and} \quad \begin{pmatrix} -1 & 1 \\ -1 & 0 \end{pmatrix},
\end{equation}
then the critical locus of $W$ over $\CC$ is $\{(\zeta, \zeta) \in (\CC^*)^2 : \zeta^3 = 1\}$.
\end{lemma}
\begin{proof}
Suppose $\HLM^+$ has even order.  Then it contains $-I$, whose $1$-eigenspace is trivial, so \cref{lemKerIntersect} tells us that any point in $H^1(L; \CC^*)$ fixed by $-I$ is a critical point of $W$.  Conversely, any critical point of $W$ must be fixed by $-I$.  Hence the critical locus of $W$ is $\{\pm 1\} \times \{\pm 1\}$.

The argument for $\lvert \HLM^+ \rvert$ divisible by $3$ is similar, using the matrices in \eqref{eqOrder3}.
\end{proof}

\begin{proof}[Proof of \cref{propMainClassification}]
First we rule out the six impossible groups, using \cref{lemExistenceOfCriticalPoints} and the fact that the action of $\HLM$ must fix each critical point of $W$ (\cref{lemLocalSystemAction}).

Suppose that $\lvert \HLM^+ \rvert$ is even.  Then $\{\pm 1\} \times \{\pm 1\}$ are critical points of $W$.  To rule out $\gp{2}_\rmt$, $\gp{4}$, and $\gp{4}_\rmd$, note that we can choose bases for $H_1(L)$ such that $\gp{2}_\rmt$ and $\gp{4} \subset \gp{4}_\rmd$ contain
\[
\begin{pmatrix} 0 & 1 \\ 1 & 0 \end{pmatrix} \quad \text{and} \quad \begin{pmatrix} 0 & -1 \\ 1 & 0 \end{pmatrix}
\]
respectively.  These elements' actions on $H^1(L; \CC^*)$ swap $(-1, 1)$ and $(1, -1)$, which is forbidden.

Now suppose instead that $\lvert \HLM^+ \rvert$ is divisible by $3$, and fix the same basis for $H_1(L)$ as in \cref{lemExistenceOfCriticalPoints}.  In the hexagon picture discussed above \cref{propMainClassification}, this basis comprises two non-adjacent, non-opposite vertices of the primitive hexagon.  The critical locus of $W$ is then $\{(\zeta, \zeta) : \zeta^3 = 1\}$.  The groups $\gp{3}_\rme$ and $\gp{6} \subset \gp{6}_\rmd$ contain
\[
\begin{pmatrix} -1 & 1 \\ 0 & 1 \end{pmatrix} \quad \text{and} \quad \begin{pmatrix} 1 & -1 \\ 1 & 0 \end{pmatrix}
\]
respectively, whose actions on $H^1(L; \CC^*)$ swap the $\zeta = e^{2\pi i /3}$ and $\zeta = e^{4\pi i/3}$ critical points.  So these groups are similarly ruled out.

Our second task is to exhibit monotone toric fibres realising the groups $\gp{1}$, $\gp{1}_\rmf$, $\gp{1}_\rmt$, $\gp{2}_\rmf$, and $\gp{3}_\rmv$.  We do this by listing suitable toric manifolds alongside pictures of the corresponding $\Delta$ in \cref{tabToricExamples}, where $\Bl_k$ denotes blowup at $k$ toric fixed points.
\begin{table}
\renewcommand{\arraystretch}{2}
\begin{tabular}{>{\centering$}m{2cm}<{$} >{\centering$}m{4cm}<{$} >{\centering\arraybackslash}m{6cm}}
\toprule
\HLM & \text{Toric manifold} & $\Delta$
\\ \midrule \addlinespace[0.5em]
\gp{1} & \Bl_2 \CC\PP^2 \text{ or } \Bl_3 \CC\PP^2 & \begin{tikzpicture}\begin{scope}[xshift=-1.7cm]\draw (-1, -1) -- (1, -1) -- (1, 0) -- (0, 1) -- (-1, 1) -- cycle;\end{scope}\draw (0,0) node[anchor=north]{or};\begin{scope}[xshift=1.7cm]\draw (-1, 0) -- (0, -1) -- (1, -1) -- (1, 0) -- (0, 1) -- (-1, 1) -- cycle;\end{scope}\end{tikzpicture}
\\ \gp{1}_\rmf & \CC \times \CC\PP^1 & \begin{tikzpicture}\draw[dashed] (2, 1) -- (1, 1);\draw (1, 1) -- (-1, 1) -- (-1, -1) -- (1, -1);\draw[dashed] (2, -1) -- (1, -1);\end{tikzpicture}
\\ \gp{1}_\rmt & \Bl_1 \CC\PP^2 \text{ or } \CC^2 & \begin{tikzpicture}\begin{scope}[xshift=-1.7cm, scale=2/3]\draw (-1, 0) -- (0, -1) -- (2, -1) -- (-1, 2) -- cycle;\end{scope}\draw (0,0.5) node[anchor=north]{or};\begin{scope}[xshift=1.7cm, scale=2/3]\draw[dashed] (-1, 2) -- (-1, 1);\draw (-1, 1) -- (-1, -1) -- (1, -1);\draw[dashed] (1, -1) -- (2, -1);\end{scope}\end{tikzpicture}
\\ \gp{2}_\rmf & \CC\PP^1 \times \CC\PP^1 & \begin{tikzpicture}\draw (-1, -1) -- (-1, 1) -- (1, 1) -- (1, -1) -- cycle;\end{tikzpicture}
\\ \gp{3}_\rmv & \CC\PP^2 & \begin{tikzpicture}[scale=2/3]\draw (-1, -1) -- (2, -1) -- (-1, 2) -- cycle;\end{tikzpicture}
\\ \bottomrule
\end{tabular}
\caption{Toric manifolds realising various groups for $\HLM$.\label{tabToricExamples}}
\end{table}

Now suppose $L$ is a monotone toric fibre.  Our final task is to explain why $\HLM$ cannot be $\gp{2}$ or $\gp{3}$.  Ruling out $\gp{3}$ is straightforward, since we know from \cref{sscToricProof} that $\HLM$ is a product of symmetric groups, so cannot have order $3$.  To rule out $\gp{2}$, suppose $\HLM$ has order $2$.  Then, again from \cref{sscToricProof}, $\HLM$ is generated by the transposition $g$ of two normal vectors, say $\nu_1$ and $\nu_2$.  If $\nu_1$ and $\nu_2$ are linearly independent over $\QQ$ then $g$ is conjugate to $g_\rmt$ in $\GL(2, \QQ)$.  Otherwise $\nu_1$ and $\nu_2$ are proportional over $\QQ$, and since they are primitive and distinct they must in fact be negatives of each other.  This means $g$ is conjugate to $g_\rmf$ in $\GL(2, \QQ)$.  In either case we see that the generator of $\HLM$ is orientation-reversing, so $\HLM$ cannot be $\gp{2}$.
\end{proof}


\subsection{Constraints on Floer theory from $\HLM$}
\label{sscFloerConstraints}

We now change perspective and start with a monotone Lagrangian $2$-torus $L$ and suppose that $\HLM$ is known.  Our goal is to extract information about the Floer theory of $L$.  We restrict to the $r=2$ case, since \cref{Theorem4}\ref{rk1} already gives our best results for $r=1$.  We also restrict to $\HLM^+ = \gp{2}$ or $\gp{3}$, since $\HLM^+ = \gp{1}$ doesn't tell us much.

First suppose that $\HLM^+ = \gp{3}$.  Up to symplectomorphism, the only monotone torus we know with $\HLM^+ = \gp{3}$ is the monotone toric fibre in $\CC\PP^2$, so one might expect $L$ to look Floer-theoretically like the latter.  It is proved in \cite[Theorem 5]{SmithSuperfiltered} that the full Floer $A_\infty$-algebra of $L$ is encoded in its superpotential $W$ (in particular, the Floer algebra is quasi-isomorphic to the endomorphism dg-algebra of a specific matrix factorisation of $W$), so one might hope to relate $W$ to $W_{\CC\PP^2}$.

We saw in \cref{lemExistenceOfCriticalPoints} that, with respect to a suitable basis of $H_1(L)$, the critical points of $W$ over $\CC$ are $\{(\zeta, \zeta) : \zeta^3 = 1\}$.  Let $x$ and $y$ be the induced coordinates on $H^1(L; \CC^*)$.  With respect to the standard coordinates, $W_{\CC\PP^2}$ is given by $x + y + \frac{1}{xy}$, which has exactly the same critical points, so we already have some evidence that $W$ and $W_{\CC\PP^2}$ are related.  The next result takes this a step further, identifying them to second order (up to an overall scale factor).  Recall that the second order behaviour of the superpotential encodes the multiplication on the Floer cohomology ring.

\begin{proposition}
\label{propFloer3}
If $\HLM^+ = \gp{3}$ then for each critical point $\calL = (\zeta, \zeta)$ there exists $\eps$ of the form $\pm \zeta^j$ such that
\[
\Hess_\calL W = \eps \Hess_\calL W_{\CC\PP^2},
\]
where $W_{\CC\PP^2} = x+y+\frac{1}{xy}$.
\end{proposition}

\begin{proof}[Proof of \cref{propFloer3}]
Fix a cube root of unity $\zeta$, and equip $L$ with the local system $\calL = (\zeta, \zeta)$ over $\ZZ[\zeta]$.  As in \cref{sscrk1proof} let $u$ and $v$ be the basis for $H^1(L)$ dual to our basis for $H_1(L)$, and let $HF^*$ denote $HF^*((L, \calL), (L, \calL))$.  Again $HF^*$ is a Clifford algebra on $u$ and $v$, satisfying $u^2 = \lambda$, $uv+vu = \mu$, and $v^2 = \nu$, where $\lambda, \mu, \nu \in \ZZ[\zeta]$ are now given by
\[
\lambda = -\frac{\zeta^2}{2}\frac{\partial^2 W}{\partial x^2}(\zeta, \zeta) \text{, \quad} \mu = -\zeta^2\frac{\partial^2 W}{\partial x \, \partial y}(\zeta, \zeta) \text{, \quad and \quad} \nu = -\frac{\zeta^2}{2}\frac{\partial^2 W}{\partial y^2}(\zeta, \zeta).
\]
We need to show that $\lambda = \mu = \nu$ is of the form $\pm \zeta^j$; then $\eps = -\lambda/\zeta$ works.

Consider the first generator of $\gp{3}$ in \eqref{eqOrder3}.  Its action on $H^1(L)$ sends $u$ to $-v$ and $v$ to $u-v$.  This must respect the ring structure on $HF^*$ (either directly by invariance of Floer cohomology or indirectly by invariance of $W$), so we must have $\lambda = \mu = \nu$.  We also know from \cref{lemConjugation} that there exists $c \in HF^0$ such that
\[
cu = -vc \quad \text{and} \quad cv = (u-v)c.
\]
This $c$ is of the form $p+quv$ for some $p,q \in \ZZ[\zeta]$, and plugging this expression into the above equalities gives $p=0$.  We also know that $c$ is inverted by some element $r+suv$, where $r,s\in\ZZ[\zeta]$, which gives $\lambda^2sq = -1$.  Therefore $\lambda$ is a unit in $\ZZ[\zeta]$, and hence is of the form $\pm \zeta^j$.
\end{proof}

Suppose instead that $\HLM^+ = \gp{2}$.  We would now like to show that $L$ Floer-theoretically resembles the monotone toric fibre in $\CC\PP^1 \times \CC\PP^1$.  From \cref{propMainClassification} we know that, with respect to any basis of $H_1(L)$, the critical points of $W$ over $\CC$ are $\{\pm 1\} \times \{\pm 1 \}$.  These are the same as the critical points of $W_{\CC\PP^1 \times \CC\PP^1} = x+\frac{1}{x}+y+\frac{1}{y}$.

\begin{proposition}
If $\HLM^+ = \gp{2}$ then for each critical point $\calL$ there exist coordinates $(x, y)$ on $H^1(L; \CC^*)$ such that either:
\[
\Hess_\calL W = \Hess_\calL \Big(xy+\frac{1}{xy}\Big),
\]
or there exist $\eps_1, \eps_2 \in \{\pm 1\}$ such that
\[
\Hess_\calL W = \Hess_\calL \Big(\eps_1\Big(x+\frac{1}{x}\Big) + \eps_2 \Big(y+\frac{1}{y}\Big)\Big).
\]
If $\HLM = \gp{2}_\rmf$ then only the second case is possible (and this occurs with $\eps_1 = \eps_2 = 1$ for the monotone toric fibre in $\CC\PP^1 \times \CC\PP^1$).
\end{proposition}
\begin{proof}
Equip $L$ with the local system $\calL = (\xi_1, \xi_2) \in \{\pm 1\} \times \{\pm 1\}$ over $\ZZ$, and consider its Floer cohomology $HF^*$.  This is generated by $u$ and $v$ subject to $u^2 = \lambda$, $uv+vu = \mu$, $v^2 = \nu$ for integers $\lambda$, $\mu$, and $\nu$.  Since $W$ is $\HLM$-invariant we have $W(x^{-1}, y^{-1}) = W(x, y)$, so $W$ is of the form
\[
W = n_{0,0} + \smashoperator{\sum_{\substack{i > 0 \\ \text{or} \\ i=0 \text{, } j > 0}}} n_{i,j} (x^iy^j + x^{-i}y^{-j})
\]
for some integers $n_{i,j}$.  We then have, using \eqref{eqCliffordRelations}, that
\[
\lambda = -\sum n_{i,j}i^2\xi_1^i\xi_2^j \text{, \quad} \mu = -2\sum n_{i,j}ij\xi_1^i\xi_2^j \text{, \quad and \quad} \nu = -\sum n_{i,j}j^2\xi_1^i\xi_2^j,
\]
where the sums are all implicitly over those $(i,j)$ with $i > 0$ or $i=0$ and $j>0$ as before.  In particular, $\mu$ must be even---say $\mu=2\mu'$.  The quadratic form $Q = -\tfrac{1}{2} \Hess_\calL W$ defining our Clifford algebra arises from the symmetric bilinear form represented by the matrix
\[
\begin{pmatrix} \lambda & \mu' \\ \mu' & \nu \end{pmatrix},
\]
and we need to show that by a change of basis we can transform it to
\[
Q_1 = \begin{pmatrix} 0 & 1 \\ 1 & 0 \end{pmatrix} \quad \text{or} \quad Q_2 = \begin{pmatrix} \eps_1 & 0 \\ 0 & \eps_2 \end{pmatrix},
\]
i.e.~to the hyperbolic form or to a diagonal unimodular form.

By \cref{lemConjugation} we know that there exist integral $c, c^{-1} \in HF^0$ such that $cu=-uc$ and $cv=-vc$.  Writing out what this means, we see that $c = a(\mu'-uv)$ and $c^{-1} = b(\mu'-uv)$ for some integers $a$ and $b$ satisfying
\[
ab(\mu'^2-\lambda\nu) = 1.
\]
In particular, $a$, $b$, and $\mu'^2-\lambda\nu$ must all be $\pm 1$.  The latter says precisely that the form is unimodular.

Suppose first that $\mu'^2-\lambda\nu = -1$.  Then $Q$ is positive or negative definite, and it is well-known that by a change of basis we can transform it to $Q_2$ with $\eps_1 = \eps_2$ equal to $1$ or $-1$ respectively.

Now suppose that $\mu'^2-\lambda\nu = 1$.  Then $Q$ is indefinite, of rank $2$ and signature $0$, so is determined up to equivalence by its parity.  If it's even then it's equivalent to the hyperbolic form, whilst if it's odd then it's equivalent to the diagonal form $Q_2$ with $\eps_1 = -\eps_2$.

Finally suppose that $\HLM = \gp{2}_\rmf$.  We may assume that our basis for $H_1(L)$ was chosen so that $\HLM$ contains $g_\rmf$.  Then $\HLM$-invariance of $Q$ implies that $\mu = 0$, and by considering continuation elements associated to the actions of $g_\rmf$ and $-g_\rmf$ we get that $\lambda$ and $\nu$ are $\pm 1$.
\end{proof}


\section{Proof of \cref{Theorem5}}
\label{secThm5}

In this section we explore the product Lagrangian tori in $\CC \times T^*S^1$ and $\CC\PP^1 \times T^*S^1$ and compute their homological Lagrangian monodromy groups.  In the latter case we obtain the first known examples of elements of infinite order in $\HLM$ for monotone $L$, independently constructed by Brendel using different methods \cite{Brendel}.


\subsection{$\CC \times T^*S^1$}

Let $X = \CC \times T^*S^1$, let $L \subset X$ be the product of the unit circle and the zero section, and let $\gamma_1, \gamma_2$ be the corresponding basis for $H_1(L)$.  We need to show that $W = x$, and $\HLM$ is trivial.

The computation of the superpotential is well-known: equip $X$ with a product complex structure that makes it diffeomorphic to $\CC \times \CC^*$, and use the open mapping principle to see that the only holomorphic discs of index $2$ are, up to reparametrisation, simply inclusions of the unit disc in $\CC \times \{\text{point}\}$.  For exactly the same reason as in the toric case, these discs are regular and count positively for the standard spin structure.  So $W = z^{\gamma_1} = x$.

It remains to compute $\HLM$, and by \cref{Theorem4}\ref{rk1} we know that every element has the form
\[
g = \begin{pmatrix} 1 & m \\ 0 & \eps \end{pmatrix}
\]
for some $m \in \ZZ$ and some $\eps \in \{\pm 1\}$.  Suppose this $g$ lies in $\HLM$, realised by a Hamiltonian isotopy $\phi_t$, and let $i_*$ denote the pushforward $H_1(L) \to H_1(X) = \langle i_*(\gamma_2) \rangle \cong \ZZ$.  Since Hamiltonian diffeomorphisms act trivially on the homology of $X$ we must have
\[
i_*(\gamma_2) = i_*(g\gamma_2) = i_*(m\gamma_1 + \eps\gamma_2) = \eps i_*(\gamma_2),
\]
and hence $\eps = 1$.

Now consider the cylinder $C$ swept out by (a curve in class) $\gamma_2$ along the isotopy $\phi_t$.  Since $\phi_t$ is Hamiltonian, the flux $\int_C \omega$ must vanish.  Because the ends of $C$ lie on the Lagrangian $L$, the expression $\int_C \omega$ depends only on the class of $C$ in $H_2(X, L)$.  In the present situation, $H_2(X, L)$ is isomorphic to $\langle \gamma_1 \rangle \cong \ZZ$ via the boundary map, and $C$ is in the class corresponding to $m\gamma_1$.  This class is also represented by $m$ times one of the holomorphic discs $u$ contributing to $W$, so
\[
0 = \int_C \omega = m \int_{D^2} u^* \omega.
\]
This forces $m$ to be zero, so $\HLM$ is trivial.

\begin{remark}
\label{rmkFlux}
Suppose $L$ is an arbitrary Lagrangian in a symplectic manifold $X$.  Given a loop $\gamma$ in $L$, and any isotopy $\phi_t$ of $X$ satisfying $\phi_1(L) = L$, the cylinder swept by $\gamma$ along $\phi_t$ defines a class $[C]$ in $\pi_1(\Omega X, \Omega L)$, and this class has a well-defined area.  Changing $\phi_t$ to a different isotopy with the same end-points modifies $[C]$ by an element of $\pi_1(\Omega X) \cong \pi_2(X)$, by the long exact sequence of the pair, so the area of $[C]$ is determined modulo $\omega(\pi_2(X))$ by the two end-points.  If $\phi_t$ is actually a Hamiltonian isotopy then $[C]$ has area zero, so the area of any cylinder in $X$ connecting $\gamma$ to $\phi_1(\gamma)$ must lie in $\omega(\pi_2(X))$.  If also $L$ is monotone then we can replace area with Maslov index, and deduce that the Maslov index of any such cylinder must lie in $2c_1(X)(\pi_2(X))$.

Suppose now that $L$ is a monotone $2$-torus with $W = a \pm x$, realising the monodromy $\left(\begin{smallmatrix} 1 & 1 \\ 0 & 1 \end{smallmatrix}\right)$.  There is an obvious cylinder from $\gamma_2$ to $\gamma_1 + \gamma_2$ given by taking the constant cylinder from $\gamma_2$ to itself and forming its boundary connected sum with an index $2$ disc with boundary $\gamma_1$ (such a disc exists because its class contributes to the superpotential).  This cylinder has Maslov index $2$, so by the previous paragraph there must be a sphere in $X$ with Chern number $1$.
\end{remark}


\subsection{$\CC\PP^1 \times T^*S^1$}

Let $X = \CC\PP^1 \times T^*S^1$ and let $L \subset X$ be the product of the equator and the zero section.  This has $W = x+\frac{1}{x}$ with respect to the obvious basis of $H_1(L)$, by a similar calculation to that for $\CC \times T^*S^1$.  By \cref{Theorem4}, $\HLM$ is a subgroup of the group generated by
\[
g = \begin{pmatrix} 1 & 2 \\ 0 & 1 \end{pmatrix} \quad \text{and} \quad h = \begin{pmatrix} -1 & 0 \\ 0 & 1 \end{pmatrix}.
\]
It's clear that $\HLM$ contains $h$, by rotating $\CC\PP^1$ about a horizontal axis.  We need to show that $\HLM$ also contains $g$.

To do this, we view $X$ as a symplectic reduction of $\CC^2 \times T^*S^1$ with respect to the diagonal $\U(1)$-action on the $\CC^2$ factor.  Let $(z, w)$ denote complex coordinates on $\CC^2$, let $\theta$ be the standard coordinate $S^1$, and let $t$ be the dual coordinate on cotangent fibres.  Let $\rho : S^1 \to \SU(2)$ be the loop
\[
\rho_\theta = \begin{pmatrix} e^{i\theta} & 0 \\ 0 & e^{-i\theta} \end{pmatrix}.
\]
Using the standard Hamiltonian $\SU(2)$-action on $\CC^2$, $\rho$ gives a loop $\psi : S^1 \to \Ham(\CC^2)$, generated by the Hamiltonian $H_\theta = \frac{1}{2}(\lvert z \rvert^2 - \lvert w \rvert^2)$.  We can then suspend $\psi$ to give a symplectomorphism $\Psi$ of $\CC^2 \times T^*S^1$, defined by
\[
\big((z, w), (t, \theta)\big) \mapsto \big(\psi_\theta(z, w), (t-H_\theta(\psi_\theta(z, w)), \theta)\big) = \Big((e^{i\theta}z, e^{-i\theta}w), \Big(t - \frac{1}{2}(\lvert z \rvert^2 - \lvert w \rvert^2), \theta\Big)\Big).
\]
This commutes with the symplectic reduction operation, which we perform at the level $\lvert z \rvert^2 + \lvert w \rvert^2 = 2$, and induces a symplectomorphism $\overline{\Psi}$ of $X$.  Note that $L$ corresponds to $\lvert z \rvert = \lvert w \vert = 1$ and $t = 0$, so is preserved by $\overline{\Psi}$.  The generators $\gamma_1$ and $\gamma_2$ of $H_1(L)$ are given respectively by rotating $z/w$ around $\U(1)$ whilst $\theta$ is constant, and by rotating $\theta$ whilst keeping $z/w$ constant.  Using this description, we see that $\overline{\Psi}$ fixes $\gamma_1$ and sends $\gamma_2$ to $2\gamma_1 + \gamma_2$, so acts precisely as $g$.  We're therefore done if we can show that $\overline{\Psi}$ is Hamiltonian isotopic to $\id_X$.

To this end, first note that $\overline{\Psi}$ is symplectically isotopic to the $\id_X$, as follows.  The group $\SU(2)$ is simply connected, so we can pick a homotopy $\rho^s$, for $s \in [0, 1]$, from $\rho^0 = \rho$ to the constant loop $\rho^1$ at the identity matrix.  Each loop $\rho^s$ is Hamiltonian, generated by some $H^s_\theta$ arising from the $\SU(2)$ moment map, and we can suspend it to a symplectomorphism $\Psi^s$ of $\CC^2 \times T^*S^1$, given by
\[
\Psi^s : \big((z, w), (t, \theta)\big) \mapsto \big(\psi^s_\theta(z, w), (t-H^s_\theta(\psi^s_\theta(z, w)), \theta)\big).
\]
Again this descends to a symplectomorphism $\overline{\Psi}^s$ of $X$, since each $\psi_\theta^s$ acts as an element of $\SU(2)$ and thus preserves $\lvert z \rvert^2 + \lvert w \rvert^2$ and commutes with the diagonal $\U(1)$-action on $\CC^2$.

To upgrade this to a Hamiltonian isotopy of $X$, we want to argue that the path $\overline{\Psi}^s$ has zero flux.  We have to be a little careful since $X$ is non-compact, and our symplectomorphisms are not compactly supported, but we can in fact work around this, as follows.  Each $\overline{\Psi}^s$ is equivariant with respect to translation in the $t$-direction, so descends to a symplectomorphism $\overline{\Psi}_\ZZ^s$ of $X/\ZZ$, where $1 \in \ZZ$ acts by $t \mapsto t+1$.  This quotient is compact, and if we can show that the flux is zero here then the path $\overline{\Psi}_\ZZ^s$ is Hamiltonian (see \cite[Theorem 10.2.5]{McDuffSalamon}).  We can then lift a generating Hamiltonian for $\overline{\Psi}_\ZZ^s$ from $X/ \ZZ$ to $X$ and see that $\overline{\Psi}^s$ itself is Hamiltonian, completing the proof.

It remains then to compute the flux of $\overline{\Psi}_\ZZ^s$, acting on $X/\ZZ = \CC\PP^1 \times T^2$.  The loop in the $t$-direction of $T^2$ is preserved setwise by the path $\overline{\Psi}_\ZZ^s$, so trivially the flux through it is zero.  We now just need to consider the loop in the $\theta$-direction.  This loop lifts to $X$ and then to $\CC^2 \times T^*S^1$, so we can compute the flux through it in the latter.  This amounts to finding the area of a specific cylinder connecting $\big((e^{i\theta}, e^{-i\theta}), (0, \theta)\big)$ to $\big((1, 1), (0, \theta)\big)$, where $\theta$ is now also the $S^1$ coordinate on the cylinder.  But since $\CC^2 \times T^*S^1$ is aspherical, the discussion in \cref{rmkFlux} means that we can actually use \emph{any} such cylinder.  There is an obvious one, namely the constant cylinder from $\big((1, 1), (0, \theta)\big)$ to itself, boundary-connected-summed with the unit disc in the first $\CC$ factor and its complex conjugate in the second.  The constant cylinder has area zero, whilst the areas of the two discs cancel out.  We conclude that the path $\overline{\Psi}_\ZZ^s$, and hence $\overline{\Psi}^s$, is Hamiltonian, so we're done.


\section{Higher dimensions}
\label{secHigher}

In this final section we return to the general setting of a monotone Lagrangian $n$-torus.


\subsection{Proof of \cref{Theorem6}}

Suppose that $n=3$ and that $\HLM$ is finite.  We can thus view $\HLM$ as a finite subgroup of $\GL(3, \ZZ)$, up to conjugation.  Our goal is to show that $\HLM$ is abstractly isomorphic to a subgroup of $S_4$, $S_3 \times S_2$, or $S_2 \times S_2 \times S_2$.

\begin{lemma}
\label{lemMinusId}
Suppose that $-I \in \HLM$. Then $\HLM$ is isomorphic to a subgroup of $S_2^3$.
\end{lemma}
\begin{proof}
First we claim that every element of $\HLM$ has order $2$.  The possible finite orders of elements in $\GL(3, \ZZ)$ are $1$, $2$, $3$, $4$, and $6$, so it suffices to rule out the existence of an element of order $3$ or $4$ in $\HLM$ (if there were an element of order $6$ then its square would be an element of order $3$).  Any such element is conjugate in $\GL(3, \ZZ)$ to
\[
\begin{pmatrix} 1 & 0 & \delta \\ 0 & 0 & -1 \\ 0 & 1 & -1 \end{pmatrix} \quad \text{or} \quad \begin{pmatrix} \pm 1 & 0 & \delta \\ 0 & 0 & -1 \\ 0 & 1 & 0 \end{pmatrix}
\]
respectively, for some $\delta \in \{0, 1\}$; see \cite[Sections 2.2 and 2.3]{AuslanderCook} for instance.  By \cref{exMinusId} and \cref{corEigenspaces}, the actions of these elements on $H^1(L; \CC^*)$ must fix every point in $\{\pm 1\}^3$.  This is easily disconfirmed, proving our claim.

So $\HLM$ is a finite group in which every element has every order $2$.  This means it is abelian and isomorphic to $S_2^k$ for some $k$.  But now we can consult the list of $32$ crystallographic point groups (finite subgroups of $\GL(3, \ZZ)$ up to conjugacy in $\GL(3, \QQ)$) and see that none of them is isomorphic to $S_2^4$.  Therefore $k \leq 3$.
\end{proof}

To complete the proof of \cref{Theorem6}, we break into two cases depending on whether $-I$ is in $\HLM$ or not.  In the first case we're done by \cref{lemMinusId}, whilst in the second case one can directly check the list of crystallographic point groups and see that all of those not containing $-I$ are isomorphic to subgroups of $S_4$ or $S_3 \times S_2$.


\subsection{Higher dimensions}
\label{sscHigherDim}

For general $n$ one can try to use similar methods to constrain the isomorphism class of $\HLM$, but the problem gets much harder because the number of conjugacy classes of finite subgroups of $\GL(n, \ZZ)$ grows rapidly: there are $85{,}308$ for $n=6$ \cite{OEIS}.  Using the computer algebra system \emph{GAP} \cite{GAP4} and the \emph{RatProbAlgTori} package \cite{RatProbAlgTori} one can enumerate these conjugacy classes of subgroups for $n = 4$, $5$, or $6$, and rule out many of them using \cref{corEigenspaces} with $k=1$ or $2$.  Our code for this is available at \url{https://github.com/MarcinAugustynowicz/HLM}.  Combining this with our earlier results for $n=2$ and $3$, and for toric fibres, we obtain the following.

\begin{proposition}
Assume $\HLM$ is finite.  If $n \leq 6$ or $L$ is a monotone toric fibre then either:
\begin{enumerate}
\item\label{conjitm1b} $\HLM$ is isomorphic to a subgroup of $\GL(n-1, \ZZ)$, or
\item\label{conjitm2b} There exist integers $n_1, \dots, n_k \geq 2$ with $\sum (n_j - 1) = n$ such that $\HLM$ is isomorphic to a subgroup of $S_{n_1} \times \dots \times S_{n_k}$.\hfill$\qed$
\end{enumerate}
\end{proposition}

This leads us to make \cref{conjHighDim}: that this result continues to hold for all monotone tori $L$.

\begin{remark}
One might hope, at first sight, that case \ref{conjitm1b} is unnecessary, since it is not needed for $n \leq 3$ or for toric $L$.  However, further thought makes this seem unlikely.  Algebraically, if all elements of $\HLM$ have a common fixed vector then \eqref{eqKerIntersect} never holds, so \cref{corEigenspaces} gives us no information.  Geometrically, meanwhile, for any finite subgroup $\Gamma \subset \GL(n-1, \ZZ)$ one can imagine a torus having superpotential of the form
\[
W(x_1, \dots, x_n) = x_1 + \widetilde{W}(x_2, \dots, x_n),
\]
where $\widetilde{W}$ is $\Gamma$-invariant.  Then it seems difficult to rule out the possibility that $\HLM = 1 \times \Gamma$ since $W$ is invariant under this group and has no critical points, so our Floer-theoretic techniques all fail.
\end{remark}

\bibliographystyle{siam} 
\bibliography{HLMbiblio} 

\end{document}